\documentclass{amsart}
\usepackage{amsmath,amssymb}
\title[Nakayama functor for coalgebras and the integral theory]{Nakayama functor for coalgebras and a categorical perspective of the integral theory for Hopf algebras}
\author[K. Shimizu]{Kenichi Shimizu}
\address[K. Shimizu]{Department of Mathematical Sciences,
  Shibaura Institute of Technology \\
  307 Fukasaku, Minuma-ku, Saitama-shi, Saitama 337-8570, Japan.}
\email{kshimizu@shibaura-it.ac.jp}
\keywords{Nakayama functor, locally finite abelian category, tensor category, coquasi-bialgebra}
\subjclass[2020]{18M05, 16T05}

\date{}

\usepackage{mathtools}
\usepackage{bbm}
\usepackage[setpagesize=false]{hyperref}
\usepackage{tikz}
\usepackage{tikz-cd}
\usetikzlibrary{calc}
\usepackage{accents}
\usepackage{amsthm}
\numberwithin{equation}{section}
\theoremstyle{plain}
\newtheorem{C}{}[section] 
\newtheorem{lemma}[C]{Lemma}

\newtheorem{theorem}[C]{Theorem}

\newtheorem{corollary}[C]{Corollary}
\newtheorem{problem}[C]{Problem}
\newtheorem{question}[C]{Question}
\newtheorem{observation}[C]{Observation}
\newtheorem{assumption}[C]{Assumption}
\theoremstyle{definition}
\newtheorem{definition}[C]{Definition}
\theoremstyle{remark}

\newcommand{\id}{\mathrm{id}}
\newcommand{\op}{\mathrm{op}}
\newcommand{\cop}{\mathrm{cop}}
\newcommand{\radj}{\mathrm{ra}}

\newcommand{\ladj}{\mathrm{la}}
\newcommand{\lladj}{\mathrm{lla}}

\newcommand{\bfk}{\Bbbk}
\newcommand{\Obj}{\mathrm{Obj}}
\newcommand{\Hom}{\mathrm{Hom}}

\newcommand{\Ker}{\mathrm{Ker}}
\newcommand{\Img}{\mathrm{Im}}

\newcommand{\coHom}{\mathrm{coHom}}

\newcommand{\rat}{\mathrm{rat}}

\newcommand{\Inj}{\mathrm{Inj}}
\newcommand{\Proj}{\mathrm{Proj}}

\newcommand{\unitobj}{\mathbf{1}}
\newcommand{\rev}{\mathrm{rev}}
\newcommand{\eval}{\mathrm{ev}}
\newcommand{\coev}{\mathrm{coev}}

\renewcommand{\L}{\mathrm{l}}
\newcommand{\R}{\mathrm{r}}

\newcommand{\Nak}{\mathbb{N}}
\newcommand{\Nakl}{\Nak^{\L}}
\newcommand{\Nakr}{\Nak^{\R}}

\newcommand{\Mod}{\mathfrak{M}}
\newcommand{\fdmod}{\Mod_{\fd}}
\newcommand{\Vect}{\mathbf{Vec}}
\newcommand{\Sets}{\mathbf{Set}}
\newcommand{\Rep}{\mathrm{Rep}}

\newcommand{\sfInj}{\mathfrak{I}_{\mathtt{sf}}}
\newcommand{\fdInj}{\mathfrak{I}_{\mathtt{fd}}}
\newcommand{\fdPro}{\mathfrak{P}_{\mathtt{fd}}}
\newcommand{\Ind}{\mathrm{Ind}}
\newcommand{\fd}{\mathtt{fd}} 
\newcommand{\qf}{\mathtt{qf}} 

\newcommand{\modobj}{\mathsf{g}}
\newcommand{\Radford}{\mathfrak{r}}

\begin{document}

\maketitle

\begin{abstract}
  We review basic properties of the Nakayama functor for coalgebras and introduce a number of applications to tensor categories.
  We also give equivalent conditions for a coquasi-bialgebra with preantipode to admit a non-zero cointegral.
\end{abstract}

\setcounter{tocdepth}{2}
\tableofcontents

\section{Introduction}

Given a finite-dimensional algebra $A$ a field $\bfk$, we denote by ${}_A \Mod_{\fd}$ the category of finite-dimensional left $A$-modules.
The endofunctor on ${}_A \Mod_{\fd}$ defined by $M \mapsto A^* \otimes_{A} M$ is called the {\em Nakayama functor} and has a prominent role in the representation theory of finite-dimensional algebras. Recently, this functor also attracts attention from a different perspective of the study of Hopf algebras, tensor categories and their applications.
This rather expository paper aims to introduce numerous applications of the Nakayama functor to tensor categories mainly based on \cite{2023arXiv230314687S,MR4560996}.

The powerfulness of the use of the Nakayama functor for the study of tensor categories was demonstrated by Fuchs, Schaumann and Schweigert in \cite{MR4042867}.
We recall that a finite abelian category is a linear category that is equivalent to ${}_A \Mod_{\fd}$ for some finite-dimensional algebra $A$, and a finite tensor category is, roughly speaking, a finite abelian category equipped with a structure of a rigid monoidal category \cite{MR3242743}.
An important fact pointed out in \cite{MR4042867} is the formula
\begin{equation}
  \label{eq:intro-coend-formula}
  A^* \otimes_{A} M = \int^{X \in {}_A \Mod_{\fd}} \Hom_A(M, X)^* \otimes_{\bfk} X
\end{equation}
of the Nakayama functor, where the integral means a coend \cite{MR1712872}.
By this formula, one can define the {\em right exact Nakayama functor} $\Nakr_{\mathcal{M}}: \mathcal{M} \to \mathcal{M}$ for a finite abelian category $\mathcal{M}$ as the same form of coend as \eqref{eq:intro-coend-formula}.
Applications of this functor given in \cite{MR4042867} is based on the following two observations:

\begin{observation}
  \label{obs:intro-1}
  By the representation theory of finite-dimensional algebras, we know that $\Nakr_{\mathcal{M}}$ is an equivalence (respectively, the identity) if and only if $\mathcal{M} \approx {}_A \Mod_{\fd}$ for some Frobenius (respectively, symmetric Frobenius) algebra $A$.
\end{observation}

\begin{observation}
  \label{obs:intro-2}
  Given a functor $F$, we denote by $F^{\ladj}$ and $F^{\radj}$ a left and a right adjoint of $F$, respectively, if it exists. By the universal property of $\Nakr_{\mathcal{M}}$ as a coend, we have an isomorphism $F \circ \Nakr_{\mathcal{A}} \cong \Nakr_{\mathcal{B}} \circ F^{\lladj}$ for a linear functor $F: \mathcal{A} \to \mathcal{B}$ between finite abelian categories such that $F^{\radj}$, $F^{\ladj}$ and $F^{\lladj} := (F^{\ladj})^{\ladj}$ exist.
\end{observation}

Some non-trivial results on finite tensor categories now easily follow.
To demonstrate, let $\mathcal{C}$ be a finite tensor category.
By applying Observation~\ref{obs:intro-2} to the functor $F = (-) \otimes X$ or $F = X \otimes (-)$, we obtain natural isomorphisms
\begin{equation}
  \label{eq:intro-ftc-Nakayama}
  {}^{\vee\vee}\!X \otimes \Nakr_{\mathcal{C}}(\unitobj)
  \cong \Nakr_{\mathcal{C}}(X) \cong \Nakr_{\mathcal{C}}(\unitobj) \otimes X^{\vee\vee}
  \quad (X \in \mathcal{C}),
\end{equation}
where $\unitobj$, $(-)^{\vee}$ and ${}^{\vee}(-)$ are the unit object, the left duality and the right duality, respectively.
Since $\mathcal{C} \approx {}_A \Mod_{\fd}$ for some Frobenius algebra $A$ \cite{MR3242743}, Observation \ref{obs:intro-1} implies that $\Nakr_{\mathcal{C}}$ is an equivalence.
Hence $\alpha := \Nakr_{\mathcal{C}}(\unitobj)$ is invertible with respect to the tensor product. In particular, the Radford $S^4$-formula
\begin{equation}
  \label{eq:Radford-S4}
  X^{\vee\vee\vee\vee} \cong \alpha^{\vee} \otimes X \otimes \alpha
  \quad (X \in \mathcal{C})
\end{equation}
follows from \eqref{eq:intro-ftc-Nakayama}.
Observation~\ref{obs:intro-1} also implies that $\mathcal{C} \approx {}_A \Mod_{\fd}$ for some symmetric Frobenius algebra $A$ if and only if there are an isomorphism $\alpha \cong \unitobj$ and a natural isomorphism $X^{\vee\vee} \cong X$ ($X \in \mathcal{C}$).

After the work of Fuchs, Schaumann and Schweigert \cite{MR4042867}, properties of the Nakayama functor are investigated and a lot of applications to finite tensor categories and their modules are found; see \cite{MR4403279,MR4586249,2021arXiv210315772S,2021arXiv210313702S,2019arXiv190400376S,2022arXiv220808203S,MR4560990,2023arXiv230206192Y}.
Now, in view of these successful, it is natural to ask:

\begin{problem}
  Define the Nakayama functor by a coend of the form like \eqref{eq:intro-coend-formula} for a class of abelian categories and establish results describing homological properties of the category like Observation~\ref{obs:intro-2}.
\end{problem}

Given a coalgebra $C$, we denote by $\Mod^C$ and $\Mod^C_{\fd}$ the category of right $C$-comodules and its full subcategory consisting of finite-dimensional objects.
A {\em locally finite abelian category} \cite{MR3242743} is a linear category that is equivalent to $\Mod^C_{\fd}$ for some coalgebra $C$.
Our work \cite{MR4560996} gives, in summary, a solution to the above problem for locally finite abelian categories and their ind-completions.
More precisely, for a coalgebra $C$, we define the right exact Nakayama functor $\Nakr_C$ by
\begin{equation*}
  \Nakr_C : \Mod^C \to \Mod^C, \quad
  \Nakr_C(M) = C \otimes_{C^*} M
  \quad (M \in \Mod^C),
\end{equation*}
where $C^*$ is the dual algebra of $C$, which naturally acts on a left (right) $C$-comodule from the right (left).
We have the coend formula
\begin{equation}
  \label{eq:intro-coend-formula-2}
  \Nakr_C(M) = \int^{X \in \Mod^C_{\fd}} \coHom^C(X, M) \otimes_{\bfk} X
  \quad (M \in \Mod^C),
\end{equation}
where $\coHom^C$ is the coHom functor (see Subsection~\ref{subsec:qf-comodules}).
Some properties of $C$ are characterized through $\Nakr_C$ as follows:

\begin{theorem}[{\it cf}. \cite{2023arXiv230314687S,MR4560996}]
  \label{thm:intro}
  For a coalgebra $C$, the following hold.
  \begin{enumerate}
  \item $C$ is right semiperfect if and only if $\Nakr_C$ induces an equivalence from the category of finite-dimensional projective right $C$-comodules to the category of injective right $C$-comodules whose socle is finite-dimensional.
    If these equivalent conditions are satisfied, then there is an isomorphism
    \begin{equation*}
      \Nakr_C(P(S)) \cong E(S)
    \end{equation*}
    for each simple right $C$-comodule $S$, where $E(S)$ and $P(S)$ are an injective hull and a projective cover of $S$, respectively.
  \item $C$ is quasi-co-Frobenius (QcF) if and only if $\Nakr_C$ is an equivalence.
  \item $C$ is co-Frobenius if and only if $C$ is QcF and $\Nakr_C$ preserves dimensions of simple right $C$-comodules.
  \item $C$ is symmetric co-Frobenius if and only if $\Nakr_C$ is the identity functor.
  \end{enumerate}
\end{theorem}

Given a category $\mathcal{A}$, we denote its ind-completion \cite{MR2182076} by $\Ind(\mathcal{A})$ (we may identify it with $\Mod^C$ when $\mathcal{A} = \Mod^C_{\fd}$ for some coalgebra $C$).
One can define the right exact Nakayama functor $\Nakr_{\Ind(\mathcal{A})} : \Ind(\mathcal{A}) \to \Ind(\mathcal{A})$ for a locally finite abelian category $\mathcal{A}$ as the same form of coend as \eqref{eq:intro-coend-formula-2}.
Unlike the finite case, there is a technical difficulty that the functor $\Nakr_C$ does not preserve $\Mod^C_{\fd}$ in general.
The full subcategory $\Mod^C_{\fd}$ is closed under $\Nakr_C$ if $C$ is semiperfect.
Thus, if $\mathcal{A} \approx \Mod^C_{\fd}$ for some semiperfect $C$, then $\Nakr_{\Ind(\mathcal{A})}$ defines an endofunctor on $\mathcal{A}$.

The above results are applied to tensor categories as follows:
In the same way as the finite case, we see that the Nakayama functor $\Nakr_{\Ind(\mathcal{C})}$ for a tensor category $\mathcal{C}$ has a similar expression as \eqref{eq:intro-ftc-Nakayama}.
An interesting result generalizing an existence criteria for non-zero integrals on Hopf algebras is that the functor $\Nakr_{\Ind(\mathcal{C})}$ is non-zero if and only if $\mathcal{C}$ is Frobenius.
For a Frobenius tensor category $\mathcal{C}$, the functor $\Nakr_{\Ind(\mathcal{C})}$ is shown to be an equivalence.
Radford's $S^4$-formula \eqref{eq:Radford-S4} for $\mathcal{C}$ is now easily established in the same way as the finite case.
Like this, the Nakayama functor is useful for generalizing results on Hopf algebras that had been proved by using integrals.
The Nakayama functor could be a categorical alternative for integrals on Hopf algebras.

\subsection{Organization of this paper}
We introduce further applications of the Nakayama functor as well as explaining the organization of this paper.
In Section \ref{sec:preliminaries}, we fix notations used throughout this paper and recall basic results on coalgebras and comodules from \cite{MR1904645,MR1786197}.

In Section~\ref{sec:naka-coalg}, for a coalgebra $C$, we introduce two endofunctors $\Nakl_C$ and $\Nakr_C$ on $\Mod^C$, which we call the left exact and the right exact Nakayama functor, respectively (Definition \ref{def:Nakayama}).
Apart from the results introduced in the above as Theorem \ref{thm:intro}, we show that $\Nakr_C$ is the zero functor if and only if the rational part of $C^*$ is zero (Lemma \ref{lem:Nakayama-C-star-rat-zero}).
We also prove the following Calabi-Yau type property:
For a semiperfect coalgebra $C$, there is a natural isomorphism
\begin{equation*}
  \Hom^C(P, M)^* \cong \Hom^C(M, \Nakr_C(P))
\end{equation*}
for $M, P \in \Mod^C_{\fd}$ with $P$ projective (Theorem~\ref{thm:Calabi-Yau-1}).

In Section~\ref{sec:cat-interpr}, we first give the coend formula of $\Nakr_C$ as in \eqref{eq:intro-coend-formula-2}, which allows us to define the Nakayama functor $\Nakr_{\Ind(\mathcal{A})}: \Ind(\mathcal{A}) \to \Ind(\mathcal{A})$ for a locally finite abelian category $\mathcal{A}$ without referencing a coalgebra $C$ such that $\mathcal{A} \approx \Mod^C_{\fd}$ (Definition \ref{def:Nakayama-locally-finite}).
We also observe that the left semiperfectness, the right semiperfectness, the semiperfectness and the QcF property of a coalgebra $C$ are characterized in terms of the category $\Mod^C_{\fd}$ (Lemmas~\ref{lem:loc-fin-ab-semiperfect}--\ref{lem:loc-fin-ab-QcF}).
With the above preparation, we interpret results in Section \ref{sec:naka-coalg} in the language of category theory.
We also give a generalization of Observation \ref{obs:intro-2} by the same way as in the finite case (Theorem \ref{thm:Nakayama-double-adj}).

In Section~\ref{sec:tensor-cat}, we first give a characterization of Frobenius tensor categories (Theorem \ref{thm:one-sided-rigidity}).
As in the finite case, we obtain the Radford $S^4$-formula for Frobenius tensor categories (Theorem \ref{thm:Radford-S4-Frob-ten-cat}).
We announce a generalization of the relationship between the Nakayama functor and modified traces pointed out in \cite{2021arXiv210315772S,2023arXiv230314687S} in the finite setting (Theorem \ref{thm:modified-trace}).
Finally, we explain applications of the Nakayama functor for exact sequences of tensor categories. An important result is that the class of Frobenius tensor categories is closed under exact sequences (Theorem \ref{thm:intro-co-Fb-2}), as asked by Natale \cite{MR4281372}.

In Section \ref{sec:coquasibialg}, we give applications of our results to coquasi-bialgebras with preantipode.
Bulacu and Caenepeel \cite{MR1995128} gave several equivalent conditions for a coquasi-Hopf algebra to be co-Frobenius.
We generalize this result to coquasi-bialgebras with preantipode (Theorem \ref{thm:coquasi-bialg-QcF-1}).
Unlike the case of coquasi-Hopf algebras, there might be a QcF coquasi-bialgebra with preantipode that is not co-Frobenius, despite that we have no such examples ({\it cf}. Theorem \ref{thm:coquasi-bialg-QcF-2}). At the end of this paper, we discuss coquasi-bialgebras arising from the de-equivariantization process \cite{MR3160718}.

\subsection{Acknowledgement}

The author thanks Taiki Shibata for discussions.
The author is supported by JSPS KAKENHI Grant Number JP20K03520.

\section{Preliminaries on coalgebras}
\label{sec:preliminaries}

\subsection{Notation}

Throughout this paper, we work over a field $\bfk$.
The symbol $\Vect$ means the category of all vector spaces (over $\bfk$).
Given a vector space $X$, we denote by $X^* := \Hom_{\bfk}(X, \bfk)$ the dual space.
The symbol like $x^*$ is sometimes used to express an element of $X^*$.
For $x^* \in X^*$ and $x \in X$, we often write $\langle x^*, x \rangle$ to mean $x^*(x) \in \bfk$.

All (co)algebras are assumed to be (co)associative and (co)unital over $\bfk$.
Given an algebra $A$, we denote by ${}_A \Mod$ and $\Mod_A$ the category of left and right $A$-modules, respectively.
Similarly, given a coalgebra $C$, we denote by ${}^C\Mod$ and $\Mod^C$ the category of left and right $C$-comodules, respectively.
The Hom functor for these four categories are written as ${}_A \Hom$, $\Hom_A$, ${}^C\Hom$ and $\Hom^C$, respectively.
We use the subscript $\fd$ to mean the full subcategory of finite-dimensional objects.
For example, $\Mod^C_{\fd}$ is the category of finite-dimensional right $C$-comodules.

The comultiplication and the counit of a coalgebra $C$ are denoted by $\Delta$ and $\varepsilon$, respectively. To express the comultiplication of $c \in C$, we adopt the Sweedler notation like $\Delta(c) = c_{(1)} \otimes c_{(2)}$.
For a left $C$-comodule $L$ and a right $C$-comodule $R$, the coactions of $C$ are expressed by Sweedler-like notation as $\ell \mapsto \ell_{(-1)} \otimes \ell_{(0)}$ ($\ell \in L$) and $r \mapsto r_{(0)} \otimes r_{(1)}$ ($r \in R$), respectively.

If $X$ is a finite-dimensional right $C$-comodule, then $X^*$ is a left $C$-comodule by the coaction determined by $x^*_{(-1)} \langle x^*_{(0)}, x \rangle = \langle x^*, x_{(0)} \rangle x_{(1)}$ for all elements $x^* \in X^*$ and $x \in X$. Similarly, the dual space of a finite-dimensional left $C$-comodule has a natural structure of a right $C$-comodule. We note that these constructions give an anti-equivalence between $\Mod^C_{\fd}$ and ${}^C\Mod_{\fd}$.

\subsection{Rational modules}
\label{subsec:rational-modules}

If $C$ is a coalgebra, then $C^*$ is an algebra with respect to the convolution product defined by $\langle f g, c \rangle = \langle f, c_{(1)} \rangle \langle g, c_{(2)} \rangle$ for $f, g \in C^*$ and $c \in C$. The algebra $C^*$ acts on a right $C$-comodule $M$ from the left by
\begin{equation*}
  c^* \rightharpoonup m = m_{(0)} \langle f, m_{(1)} \rangle
  \quad (c^* \in C^*, m \in M)
\end{equation*}
and this construction gives rise to a fully faithful functor
\begin{equation}
  \label{eq:C-comod-inclusion}
  \Mod^C \to {}_{C^*}\Mod,
  \quad M \mapsto (M, \rightharpoonup).
\end{equation}
The image of this functor is the category of {\em rational} modules.
Every left $C^*$-module $M$ has the largest rational submodule, which is called {\em the rational part} of $M$ and denoted by $M^{\rat}$. We may, and do, regard the assignment $M \mapsto M^{\rat}$ as a functor $(-)^{\rat} : {}_{C^*}\Mod \to \Mod^C$, which is right adjoint to \eqref{eq:C-comod-inclusion}.
The rational part functor for left $C^*$-modules is defined in a similar way.

In general, computing the rational part is a non-trivial problem.
Here we note the following formula:
If $M$ is a left $C$-comodule, then $M$ has a natural structure of a right $C^*$-module and hence its dual $M^*$ is a left $C^*$-module.
By \cite[Corollary 2.2.16]{MR1786197}, the rational part of $M^*$ is given by
\begin{equation}
  \label{eq:star-rat}
  M^{*\rat} = \{ m^* \in M^* \mid \dim_{\bfk}(C^* m^*) < \infty \}.
\end{equation}

\subsection{Quasi-finite comodules and the coHom functor}
\label{subsec:qf-comodules}

Let $C$ be a coalgebra.
For each right $C$-comodule $M$, there is a linear functor
\begin{equation}
  \label{eq:comodule-copower}
  \Vect \to \Mod^C,
  \quad W \mapsto W \otimes_{\bfk} M.
\end{equation}
If the functor \eqref{eq:comodule-copower} has a left adjoint, then we denote it by $\coHom^C(M, -)$. Thus, by definition, we have a natural isomorphism
\begin{equation*}
  \Hom_{\bfk}(\coHom^C(M, N), W) \cong \Hom^C(N, W \otimes_{\bfk} M)
\end{equation*}
for $N \in \Mod^C$ and $W \in \Vect$. In particular, by letting $W = \bfk$, we have
\begin{equation}
  \label{eq:coHom-formula-1}
  \coHom^C(M, N)^* \cong \Hom^C(N, M).
\end{equation}

The functor $\coHom^C(M, -)$ is introduced by Takeuchi \cite{MR472967} in his study of Morita theory for coalgebras.
According to Takeuchi \cite{MR472967}, the functor \eqref{eq:comodule-copower} has a left adjoint if and only if $M$ is {\em quasi-finite} in the sense that the vector space $\Hom(X, M)$ is finite-dimensional for all objects $X \in \Mod^C_{\fd}$. Let $\Mod^C_{\qf}$ denote the full subcategory of $\Mod^C$ consisting of quasi-finite comodules. Then we have a functor
\begin{equation*}
  \coHom^C : (\Mod^C_{\qf})^{\op} \times \Mod^C \to \Vect,
  \quad (M, N) \mapsto \coHom^C(M, N),
\end{equation*}
which we call the {\em coHom functor}.

In \cite{MR472967}, the coHom functor is given as a certain filtered colimit.
For practical use, the following formula is useful:
We have a natural isomorphism
\begin{equation}
  \label{eq:coHom-for-fd}
  \coHom^C(X, Y) \cong X^* \otimes_{C^*} Y
  \quad (X \in \Mod^C_{\fd}, Y \in \Mod^C)
\end{equation}
by the tensor-Hom adjunction.

\subsection{Injective and projective comodules}

We recall that an object $P$ of a category $\mathcal{A}$ is said to be {\em projective} if the functor $\Hom_{\mathcal{A}}(P, -) : \mathcal{A} \to \Sets$ preserves epimorphisms, where $\Sets$ is the category of all sets. An object of $\mathcal{A}$ is said to be {\em injective} if it is projective in $\mathcal{A}^{\op}$. To clarify, we introduce the following terminology:

\begin{definition}
  A right $C$-comodule $X$ is said to be {\em projective} ({\em injective}) if it is a projective (injective) object of the category $\Mod^C$.
\end{definition}

Given a category $\mathcal{A}$, we denote by $\Proj(\mathcal{A})$ and $\Inj(\mathcal{A})$ the full subcategory of $\mathcal{A}$ consisting of projective and injective objects of $\mathcal{A}$, respectively.
In this paper, we are interested in $\Proj$ or $\Inj$ of $\Mod^C_{\fd}$ rather than $\Mod^C$.
For an algebra $A$, an object of $\Proj({}_A \Mod_{\fd})$ may not be projective in ${}_A \Mod$.
However, the situation is easy for coalgebras:

\begin{lemma}
  \label{lem:fd-proj-comodules}
  A finite-dimensional right $C$-comodule is projective (injective) in $\Mod^C$ if and only if it is projective (injective) in $\Mod^C_{\fd}$.
\end{lemma}

This lemma can be obtained by assembling some fundamental results on projective and injective comodules found in \cite{MR1786197} (see \cite[Lemma 2.3]{MR4560996} for details).

\subsection{Injective comodules and idempotents}
\label{sec:inj-comod-idempo}

Let $C$ be a coalgebra. An {\em idempotent} for a coalgebra $C$ is an idempotent of the algebra $C^*$.
Given an idempotent $e$ for $C$ and a right $C$-comodule $M$, we write $e M := \{ e \rightharpoonup m \mid m \in M \}$. A similar notation will be used for left $C$-comodules.
We note that $e C$ and $C e$ are a left and a right $C$-comodule, respectively.
It is easy to verify that the map
\begin{equation}
  \label{eq:idempotent-iso}
  \Hom^C(M, W \otimes_{\bfk} C e)
  \to \Hom_{\bfk}(e M, W),
  \quad f \mapsto (\id_W \otimes \varepsilon|_{C e}) \circ (f|_{e M})
\end{equation}
is a natural isomorphism for $M \in \Mod^C$ and $W \in \Vect$. Hence we have
\begin{equation}
  \label{eq:idempotent-iso-2}
  \coHom^C(C e, M) \cong e M
  \quad (M \in \Mod^C)
\end{equation}
as vector spaces. When $M$ is moreover a $B$-$C$-bicomodule for some coalgebra $B$, \eqref{eq:idempotent-iso-2} is in fact an isomorphism of left $B$-comodules.

According to \cite[Proposition 1.13]{MR1904645}, an indecomposable injective left (right) $C$-comodule is isomorphic to $e C$ ($C e$) for some primitive idempotent $e$ for $C$.
Let $\sfInj^C$ be the full subcategory of $\Mod^C$ consisting of injective objects whose socle is of finite length.
Since an indecomposable injective comodule is the same thing as an injective hull of a simple comodule, $\sfInj^C$ coincides with the full subcategory of $\Mod^C$ consisting of finite direct sums of injective hulls of simple comodules.
We define the full subcategory ${}^C\sfInj \subset {}^C\Mod$ in a similar way.
Then we have:

\begin{lemma}
  \label{lem:sfinj-equivalence}
  Let $C$ be a coalgebra. The contravariant functor
  \begin{equation}
    \label{eq:sfinj-equivalence}
    D: \sfInj^C \to {}^C\sfInj,
    \quad E \mapsto \coHom^C(E, C)
  \end{equation}
  is a well-defined anti-equivalence sending $C e$ to $e C$, where $e$ is a primitive idempotent.
  Moreover, the anti-equivalence \eqref{eq:sfinj-equivalence} sends an injective hull of a simple right $C$-comodule $S$ to an injective hull of $S^* \in {}^C\Mod$.
\end{lemma}
\begin{proof}
  By \eqref{eq:idempotent-iso-2}, we have an isomorphism $D(C e) \cong e C$ for each idempotent $e$ for $C$ (as noted in \cite[Corollary 1.6]{MR1904645}). This shows that the functor $D$ induces a bijective correspondence between the isomorphism classes of objects of $\sfInj^C$ and those of ${}^C\sfInj$. For idempotents $e$ and $f$ for $C$, we have isomorphisms
  \begin{equation}
    \label{eq:idempotent-iso-3}
    \Hom^C(C e, C f) \cong (e C f)^* \cong {}^C\Hom(f C, e C)
    \cong \Hom^C(D(C f), D(C e))
  \end{equation}
  by \eqref{eq:idempotent-iso} and its variant for left comodules.
  One can check that \eqref{eq:idempotent-iso-3} coincides with the linear map induced by the functor $D$.
  Thus $D$ is fully faithful.
  We have proved that $D$ is an anti-equivalence.

  Now let $S$ be a simple right $C$-comodule, and let $e$ be a primitive idempotent for $C$ such that $C e$ is an injective hull of $S$.
  Since $e C$ ($\cong D(C e)$) is an indecomposable injective left $C$-comodule, it is an injective hull of a simple left $C$-comodule.
  By the isomorphism \eqref{eq:idempotent-iso}, we have
  \begin{equation*}
    {}^C\Hom(S^*, D(C e))
    \cong {}^C\Hom(S^*, e C)
    \cong (S^* e)^* \cong (e S)^{**} \cong \Hom^C(S, C e)^* \ne 0,
  \end{equation*}
  and therefore $D(C e)$ is an injective hull of $S^*$.
\end{proof}

\subsection{Special classes of coalgebras}
\label{subsec:special-classes-of-coalgebras}

Let $C$ be a coalgebra.
We recall from \cite{MR1786197} the following classes of coalgebras:
\begin{enumerate}
\item $C$ is said to be {\em left semiperfect} if every finite-dimensional left $C$-comodule has a projective cover in ${}^C\Mod$.
\item $C$ is said to be {\em left quasi-co-Frobenius} (or {\em left QcF} for short) if there are a cardinal $\kappa$ and an injective homomorphism $C \to (C^{*})^{\oplus \kappa}$ of left $C^*$-modules.
  If, moreover, the cardinal $\kappa$ can be taken to be 1, then $C$ is said to be {\em co-Frobenius}.
\item For an adjective $\mathrm{X} \in \{ \text{semiperfect, quasi-co-Frobenius, co-Frobenius} \}$, we say that $C$ is right X if $C^{\cop}$ is left X.
  An X coalgebra means a coalgebra that is both left X and right X.
\end{enumerate}
The book \cite{MR1786197} is a comprehensive reference for results on these classes of coalgebras.
One of important results is that a left (right) QcF coalgebra is left (right) semiperfect \cite[Corollary 3.3.6]{MR1786197}.

\subsection{Local units for $C^{*\rat}$}
\label{subsec:local-units}

Given a coalgebra $C$, we denote by $C^{*\rat}_{\ell}$ and $C^{*\rat}_r$ the rational part of $C^*$ as a left and a right $C^*$-module, respectively.
By \eqref{eq:star-rat}, both $C^{*\rat}_{\ell}$ and $C^{*\rat}_r$ are ideals of the algebra $C^*$.

We recall a technique of {\em local units} for semiperfect coalgebras.
From now on, we assume that $C$ is semiperfect.
Since $C^{*\rat}_{\ell} = C^{*\rat}_r$ by the semiperfectness \cite[Corollary 3.2.16]{MR1786197}, we write $C^{*\rat}_{\ell}$ and $C^{*\rat}_r$ simply as $C^{*\rat}$.
By \eqref{eq:star-rat}, the counit of $C$ belongs to $C^{*\rat}$ only if $C$ is finite-dimensional.
Thus, in general, $C^{*\rat}$ is a `non-unital' ring.
However, it is known that $C^{*\rat}$ has local units in the following sense:
For any finite subset $X \subset C^{*\rat}$, there is an idempotent $e \in C^{*\rat}$ such that $e x = x = x e$ for all $x \in X$ \cite[Corollary 3.2.17]{MR1786197}.

We note that local units can be extended for comodules as follows:

\begin{lemma}
  \label{lem:local-unit-1}
  Let $C$ be a semiperfect coalgebra, and let $M$ be a right $C$-comodule.
  For any finite subsets $X \subset C^{*\rat}$ and $Y \subset M$, there is an idempotent $e \in C^{*\rat}$ such that $e x = x = x e$ and $e y = y$ for all $x \in X$ and $y \in Y$.
\end{lemma}
\begin{proof}
  Let $N$ be a finite-dimensional subcomodule of $M$ such that $Y \subset N$, and let $D$ be the subspace spanned by $\langle n^*, n_{(0)} \rangle n_{(1)}$ ($n^* \in N^*$, $n \in N$). It is easy to see that $D$ is a finite-dimensional subcoalgebra of $C$.
  The semiperfectness of $C$ implies that $C^{*\rat}$ is dense in $C^*$ with respect to the finite topology \cite[Proposition 3.2.1]{MR1786197}.
  Hence there is an element $\varepsilon' \in C^{*\rat}$ such that $\varepsilon'(d) = \varepsilon(d)$ for all $d \in D$. By the construction, we have
  \begin{equation*}
    \langle n^*, \varepsilon' n \rangle
    = \varepsilon'(\langle n^*, n_{(0)} \rangle n_{(1)})
    = \varepsilon(\langle n^*, n_{(0)} \rangle n_{(1)})
    = \langle n^*, n \rangle
  \end{equation*}
  for all $n^* \in M^*$ and $n \in M$. Hence we have
  \begin{equation}
    \label{eq:local-unit-1-eq-1}
    \varepsilon' n = n \quad (n \in N).
  \end{equation}
  Now let $e \in C^{*\rat}$ be a local unit for the finite set $X \cup \{ \varepsilon' \}$. The element $e$ fulfills the requirements:
  Indeed, it is trivial that the equations $e x = x = x e$ hold for all $x \in X$.
  By \eqref{eq:local-unit-1-eq-1}, we have $e n = e (\varepsilon' n) = (e \varepsilon') n = \varepsilon' n = n$ for all $n \in N$. Hence, in particular, we have $e y = y$ for all $y \in Y$. The proof is done.
\end{proof}

By using this lemma, we prove the following technical result:

\begin{lemma}
  \label{lem:local-unit-2}
  Suppose that $C$ is semiperfect. Then the map
  \begin{equation*}
    \kappa_M: C^{*\rat} \otimes_{C^*} M \to M,
    \quad \kappa_M(c^* \otimes_{C^*} m) = c^* m
  \end{equation*}
  is a natural isomorphism for $M \in \Mod^C$.
\end{lemma}
\begin{proof}
  We first consider the case where $M$ is finite-dimensional.
  By applying Lemma~\ref{lem:local-unit-1} to a basis of $M$, we obtain an idempotent $e \in C^{*\rat}$ such that $e m = m$ for all $m \in M$.
  The space $X := C^{*\rat} e$ is finite-dimensional by \eqref{eq:star-rat}.
  Again by Lemma~\ref{lem:local-unit-1}, we get an idempotent $\tilde{e} \in C^{*\rat}$ such that $\tilde{e} x = x = x \tilde{e}$, $\tilde{e} e = e = e \tilde{e}$ and $\tilde{e} m = m$ for all $x \in X$ and $m \in M$.
  Now we define the linear map
  \begin{equation*}
    \overline{\kappa}_M: M \to C^{*\rat} \otimes_{C^*} M,
    \quad \overline{\kappa}_M(m) = \tilde{e} \otimes_{C^*} m.
  \end{equation*}
  It is obvious that $\kappa_M \circ \overline{\kappa}_M$ is the identity.
  We also have
  \begin{gather*}
    \overline{\kappa}_M \kappa_M(c^* \otimes_{C^*} m)
    = \tilde{e} \otimes_{C^*} c^* m
    = \tilde{e} \otimes_{C^*} c^* (e m)
    = \tilde{e} \otimes_{C^*} (c^* e) m \\
    = \tilde{e} (c^* e) \otimes_{C^*} m
    = (c^* e) \otimes_{C^*} m
    = c^* \otimes_{C^*} e m = c^* \otimes_{C^*} m
  \end{gather*}
  for $c^* \in C^{*\rat}$ and $m \in M$, where the fifth equality follows from $c^* e \in X$. Hence we have proved that $\kappa_M$ is an isomorphism if $M$ is finite-dimensional.
  The general case follows from the fundamental theorem for comodules.
\end{proof}

\section{Nakayama functor for coalgebras}
\label{sec:naka-coalg}

\subsection{Nakayama functor for coalgebras}

Let $C$ be a coalgebra.

\begin{definition}
  \label{def:Nakayama}
  We define the left exact Nakayama functor $\Nakl_C$ and the right exact Nakayama functor $\Nakr_C$ to be the endofunctors on $\Mod^C$ given by
  \begin{equation*}
    \Nakl_C(M) = \Hom^C(C, M)^{\rat}
    \quad \text{and} \quad
    \Nakr_C(M) = C \otimes_{C^*} M,
  \end{equation*}
  respectively, for $M \in \Mod^C$. Here, the right $C$-comodule structure of $\Nakl_C(M)$ and $\Nakr_C(M)$ are given as follows:
  \begin{enumerate}
  \item The vector space $\Hom^C(C, M)$ becomes a left $C^*$-module by the action given by $(c^* \cdot f)(c) = f(c \leftharpoonup c^*)$ for $c^* \in C^*$, $f \in \Hom^C(C, M)$ and $c \in C$. We take its rational part to obtain the right $C$-comodule $\Nakl_C(M)$.
  \item The right $C$-coaction on $\Nakr_C(M)$ is given by
    \begin{equation*}
      c \otimes_{C^*} m \mapsto (c_{(1)} \otimes_{C^*} m) \otimes c_{(2)}
      \quad (c \in C, m \in M).
    \end{equation*}
  \end{enumerate}
\end{definition}

As their name suggest, $\Nakl_C$ and $\Nakr_C$ are left and right exact, respectively. This follows from Lemma~\ref{lem:Nakayama-adjoint} below and the fact that a functor is left (right) exact if it has a left (right) adjoint.

\begin{lemma}
  \label{lem:Nakayama-adjoint}
  $\Nakl_C$ is right adjoint to $\Nakr_C$.
\end{lemma}
\begin{proof}
  For $M, M' \in \Mod^C$, we have natural isomorphisms
  \begin{align*}
    \Hom^C(\Nakr_C(M), M')
    & \cong {}_{C^*}\Hom(C \otimes_{C^*} M, M') \\
    & \cong {}_{C^*}\Hom(M, {}_{C^*}\Hom(C, M')) \\
    & \cong \Hom^C(M, \Hom(C, M')^{\rat})
      = \Hom^C(M, \Nakl_C(M')),
  \end{align*}
  where the second isomorphism is the tensor-Hom adjunction.
\end{proof}

The restriction of $\Nakl_C$ to $\Mod^C_{\qf}$ has the following expression:

\begin{lemma}
  \label{lem:Nakayama-l-coHom}
  There is a natural isomorphism
  \begin{equation*}
    \Nakl_C(M) \cong \coHom^C(M, C)^{*\rat}
    \quad (M \in \Mod^C_{\qf}).
  \end{equation*}
\end{lemma}

Chin and Simson \cite[Equation 1.12]{MR2684139} have introduced a functor $\nu_C$ for a coalgebra $C$ and called it the Nakayama functor.
The functor $\nu_C$ is defined on a particular class of quasi-finite comodules.
By Lemma~\ref{lem:Nakayama-l-coHom} and \cite[Theorem 1.8 (a)]{MR2684139}, we see that $\nu_C$ is isomorphic to $\Nakl_C$ on the domain of $\nu_C$.

\begin{proof}
  By \eqref{eq:coHom-formula-1} and the defining formula of $\Nakl_C(M)$, we establish an isomorphism of vector spaces as stated.
  One can check that it is indeed an isomorphism of right $C$-comodules.
\end{proof}

The restriction of $\Nakr_C$ to $\Mod^C_{\fd}$ has the following expression:

\begin{lemma}
  \label{lem:Nakayama-r-coHom}
  For $M \in \Mod_{\fd}^C$, there are natural isomorphism
  \begin{equation}
    \Nakr_C(M) \cong \coHom^{C^{\cop}}(M^*, C^{\cop}), \quad
    \Nakl_{C^{\cop}}(M^*) \cong \Nakr_C(M)^{*\rat}.
  \end{equation}
\end{lemma}
\begin{proof}
  The first isomorphism follows from \eqref{eq:coHom-for-fd}.
  The second one is obtained from the first one by applying the functor $(-)^{*\rat}$ to the both sides.
\end{proof}

Two coalgebras $C$ and $D$ are said to be {\em Morita-Takeuchi equivalent} if $\Mod^C$ and $\Mod^D$ are equivalent as linear categories.
The Nakayama functor is Morita-Takeuchi invariant in the following sense:

\begin{lemma}
  \label{lem:Nakayama-MT-invariance}
  Let $C$ and $D$ be coalgebras, and let $\Phi : \Mod^C \to \Mod^D$ be an equivalence of linear categories. Then there are isomorphisms
  \begin{equation*}
    \Phi \circ \Nakl_C \cong \Nakl_D \circ \Phi
    \quad \text{and} \quad
    \Phi \circ \Nakr_C \cong \Nakr_D \circ \Phi.
  \end{equation*}
\end{lemma}

We will give a proof of this lemma in Subsection \ref{subsec:Nakayama-uni-pro} after recalling a categorical notion of coends. In fact, this lemma is trivial once the universal property of the Nakayama functor is formulated and proved in Subsection \ref{subsec:Nakayama-uni-pro}.

\subsection{The Nakayama functor and the rational part of $C^*$}

Let $C$ be a coalgebra.
As in Subsection~\ref{subsec:local-units}, we denote by $C^{*\rat}_{\ell}$ the rational part of $C^*$ as a left $C^*$-module.
One of main purposes of this section is to characterize some homological properties of $C$ in terms of the Nakayama functor.
In preparation for this, we first observe:

\begin{lemma}
  \label{lem:Nakayama-C-star-rat-zero}
  For a coalgebra $C$, the following are equivalent:
  \begin{enumerate}
  \item $C^{*\rat}_{\ell} = 0$.
  \item $\Nakl_C$ is the zero functor.
  \item $\Nakr_C$ is the zero functor.
  \end{enumerate}
\end{lemma}
\begin{proof}
  The tensor-Hom adjunction gives natural isomorphisms
  \begin{equation}
    \label{eq:Nakayama-dual}
    \Nakr_C(M)^*
    \cong {}_{C^*}\Hom(M, \Hom_{\bfk}(C, \bfk))
    \cong \Hom^C(M, C^{*\rat}_{\ell})
  \end{equation}
  for $M \in \Mod^C$.
  Now (1) $\Leftrightarrow$ (3) is obvious.
  The equivalence (2) $\Leftrightarrow$ (3) follows from Lemma \ref{lem:Nakayama-adjoint}, which states that $\Nakr_C$ and $\Nakl_C$ are an adjoint pair.
\end{proof}

The natural isomorphism \eqref{eq:Nakayama-dual} is also used to prove:

\begin{lemma}
  \label{lem:Nakayama-C-star-rat-qf}
  For a coalgebra $C$, the following are equivalent:
  \begin{enumerate}
  \item The right $C$-comodule $C^{*\rat}_{\ell}$ is quasi-finite.
  \item $\Nakl_C$ preserves the class of quasi-finite comodules.
  \item $\Nakr_C$ preserves the class of finite-dimensional comodules.
  \end{enumerate}
\end{lemma}
\begin{proof}
  The implication (1) $\Rightarrow$ (3) follows immediately from \eqref{eq:Nakayama-dual}.
  We suppose that (3) holds and let $M$ be a quasi-finite right $C$-comodule.
  Then we have
  \begin{equation*}
    \dim_{\bfk} \Hom^C(X, \Nakl_C(M))
    = \dim_{\bfk} \Hom^C(\Nakr_C(X), M) < \infty
  \end{equation*}
  for all $X \in \Mod^C_{\fd}$.
  This means that $\Nakl_C(M)$ is quasi-finite.
  Thus (2) holds.
  Finally, we assume that (2) holds.
  Then $C^{*\rat}_{\ell}$, which is isomorphic to $\Nakl_C(C)$, is quasi-finite since $C$ is. Hence (1) holds. The proof is done.
\end{proof}

\begin{lemma}
  \label{lem:Nakayama-C-star-rat-inj}
  For a coalgebra $C$, the following are equivalent:
  \begin{enumerate}
  \item The right $C$-comodule $C^{*\rat}_{\ell}$ is injective.
  \item $\Nakl_C$ preserves the class of injective comodules.
  \item $\Nakr_C$ is exact.
  \end{enumerate}
\end{lemma}
\begin{proof}
  The implication (1) $\Rightarrow$ (3) follows from \eqref{eq:Nakayama-dual} and the fact that the duality functor preserves and reflects exact sequences.
  The implication (3) $\Rightarrow$ (2) follows from Lemma \ref{lem:Nakayama-adjoint}.
  If (2) holds, then $\Nakl_C(C) \cong C^{*\rat}_{\ell}$ is injective since $C$ is, and thus (1) holds.
  The proof is done.
\end{proof}

\subsection{Semiperfect coalgebras}

The Nakayama functor $A^* \otimes_A (-)$ for a finite-dimensional algebra $A$ induces an equivalence from $\Proj({}_A\Mod_{\fd})$ to $\Inj({}_A \Mod_{\fd})$.
We show that an analogous result holds for a coalgebra $C$ if and only if $C$ is right semiperfect.
To be precise, we define the full subcategory $\sfInj^C \subset \Mod^C$ as in Subsection~\ref{sec:inj-comod-idempo} and write $\fdPro^C = \Proj(\Mod^C_{\fd})$ for brevity. Then we have:

\begin{theorem}
  \label{thm:Nakayama-right-semiperfect}
  For a coalgebra $C$, the following are equivalent:
  \begin{enumerate}
  \item $C$ is right semiperfect.
  \item The functor $\Nakl_C$ restricts to an equivalence from $\sfInj^C$ to $\fdPro^C$.
  \item The functor $\Nakr_C$ restricts to an equivalence from $\fdPro^C$ to $\sfInj^C$.
  \end{enumerate}
  Suppose that these equivalent conditions are satisfied.
  Then the equivalences between $\sfInj^C$ and $\fdPro^C$ induced by $\Nakl_C$ and $\Nakr_C$ are mutually quasi-inverse to each other.
  Moreover, for a simple right $C$-comodule $S$, there are isomorphisms
  \begin{equation}
    \label{eq:Nakayama-inj-hull}
    \Nakl_C(E(S)) \cong P(S)
    \quad \text{and} \quad
    \Nakr_C(P(S)) \cong E(S),
  \end{equation}
  where $E(S)$ and $P(S)$ are an injective hull and a projective cover of $S$, respectively.
\end{theorem}

We note that $\Nakl_C$ does not even define a functor from $\sfInj^C$ to $\fdPro^C$ in general (see \cite[Example 3.29]{MR4560996}). The condition (2) of this theorem means that $\Nakl_C(E)$ belongs to $\fdPro^C$ for all objects $E \in \sfInj^C$ and, moreover, the induced functor $\Nakl_C: \sfInj^C \to \fdPro^C$ is an equivalence. The same remark applies to the condition (3).

\begin{proof}
  We first show that (1) implies both (2) and (3).
  For brevity, we write ${}^C\fdInj = \Inj({}^C\Mod_{\fd})$.
  Suppose that (1) holds.
  Then, since every simple left $C$-comodule has a finite-dimensional injective hull \cite[Theorem 3.2.3]{MR1786197}, we have ${}^C\sfInj = {}^C\fdInj$.
  The assignment $X \mapsto X^*$ induces an anti-equivalence between ${}^C\fdInj$ and $\fdPro^C$. By Lemmas~\ref{lem:sfinj-equivalence} and \ref{lem:Nakayama-l-coHom}, we see that the restriction of $\Nakl_C$ to $\sfInj^C$ factors into the composition
  \begin{equation*}
    \Nakl_C : \sfInj^C
    \xrightarrow{\quad \coHom(-,C) \quad} {}^C\sfInj = {}^C\fdInj
    \xrightarrow{\quad (-)^* \quad} \fdPro^C
  \end{equation*}
  of anti-equivalences. Hence (2) holds.
  To show (3), we note that the above argument shows that every object of $\fdPro^C$ is a finite direct sum of objects of the form $(C e)^*$ for some primitive idempotent $e$ for $C$.
  Since we have
  \begin{equation}
    \label{eq:Nakayama-r-of-Ce-dual}
    \Nakr_C((C e)^*)
    = C \otimes_{C^*} (C e)^*
    = C \otimes_{C^*} C^* e
    \cong C e
  \end{equation}
  for a primitive idempotent $e$, the functor $\Nakr_C$ induces a functor from $\fdPro^C$ to $\sfInj^C$.
  It follows from Lemma \ref{lem:Nakayama-adjoint} that the induced functor $\Nakr_C : \fdPro^C \to \sfInj^C$ is left adjoint to $\Nakl_C : \sfInj^C \to \fdPro^C$. As the latter one is an equivalence as we have proved, so is the former one. Namely, (3) holds.

  Next, we prove that (2) implies (1). Suppose that (2) holds.
  Then there is an anti-equivalence $D: {}^C\sfInj \to \fdPro^C$ obtained by composing the quasi-inverse of \eqref{eq:sfinj-equivalence} and the equivalence $\Nakl_C : \sfInj^C \to \fdPro^C$.
  By Lemma \ref{lem:Nakayama-l-coHom}, we have $D(E) = E^{*\rat}$ for $E \in {}^C\sfInj$.
  We now consider the contravariant functor $\overline{D} : \fdPro^C \to {}^C\sfInj$ defined by $\overline{D}(P) = P^*$.
  It is easy to see that $D \circ \overline{D}$ is isomorphic to the identity functor on $\fdPro^C$.
  Since $D$ is an anti-equivalence, so is $\overline{D}$.
  This implies that every object of ${}^C\sfInj$ is finite-dimensional.
  By a characterization of right semiperfect coalgebras \cite[Theorem 3.2.3]{MR1786197}, we conclude that $C$ is semiperfect.

  One can prove that (3) implies (1) in a similar way:
  Suppose that (3) holds.
  Then, by applying Lemma \ref{lem:sfinj-equivalence} to $C^{\cop}$, we obtain an anti-equivalence
  \begin{equation}
    \label{eq:sfinj-equivalence-cop}
    \coHom^{C^{\cop}}(-, C^{\cop}) : {}^C\sfInj \to \sfInj^C
  \end{equation}
  of categories. Let $D: \fdPro^C \to {}^C\sfInj$ be the anti-equivalence obtained by composing the equivalence $\Nakr_C : \fdPro^C \to \sfInj^C$ and the quasi-inverse of \eqref{eq:sfinj-equivalence-cop}. By Lemma \ref{lem:Nakayama-r-coHom}, we have $D(E) \cong E^*$ for $E \in \fdPro^C$.
  This implies that every object of ${}^C\sfInj$ is finite-dimensional.
  Hence $C$ is semiperfect.

  We have proved that (1), (2) and (3) are equivalent.
  To complete the proof, we now assume that these equivalent conditions are satisfied.
  Then $\Nakl_C : \sfInj^C \to \fdPro^C$ and $\Nakr_C : \fdPro \to \sfInj^C$ are mutually quasi-inverse to each other as we have seen in the proof of (1) $\Rightarrow$ (3).
  By Lemmas~\ref{lem:sfinj-equivalence} and \ref{lem:Nakayama-l-coHom}, we have $\Nakl_C(E(S)) \cong E(S^*)^* \cong P(S)$ for simple $S \in \Mod^C$. The latter isomorphism in \eqref{eq:Nakayama-inj-hull} is obtained by applying the equivalence $\Nakr_C : \fdPro \to \sfInj^C$ to the former one. The proof is done.
\end{proof}

For our applications to locally finite abelian categories, it is important to know when the full subcategory $\Mod^C_{\fd}$ is closed under $\Nakl_C$ and $\Nakr_C$.
With the help of the above theorem, we now prove:

\begin{theorem}[{\cite[Theorem 3.15]{MR4560996}}]
  \label{thm:Nakayama-semiperfect}
  Let $C$ be a semiperfect coalgebra. Then the adjunction $\Nakr_C \dashv \Nakl_C$ of Lemma \ref{lem:Nakayama-adjoint} restricts to an adjunction on $\Mod^C_{\fd}$. Moreover, the functors $\Nakl_C$ and $\Nakr_C$ induce equivalences
  \begin{equation}
    \label{eq:Nakayama-semiperfect-equiv}
    \Nakl_C : \fdInj^C \to \fdPro^C
    \text{\quad and \quad}
    \Nakr_C : \fdPro^C \to \fdInj^C,
  \end{equation}
  which are mutually quasi-inverse to each other.
\end{theorem}
\begin{proof}
  The semiperfect of $C$ implies $\sfInj^C = \fdInj^C$. Thus, by Theorem \ref{thm:Nakayama-right-semiperfect}, we have well-defined equivalences \eqref{eq:Nakayama-semiperfect-equiv}.
  It remains to show that $\Mod^C_{\fd}$ is closed under $\Nakl_C$ and $\Nakr_C$.
  Let $M$ be a finite-dimensional right $C$-comodule.
  By the semiperfectness of $C$, there are objects $E \in \fdInj^C$ and a monomorphism $M \hookrightarrow E$ of right $C$-comodules. Since $\Nakl_C$ is left exact, $\Nakl_C(M)$ is finite-dimensional as a subcomodule of $\Nakl_C(M) \in \fdPro^C$.
  A dual argument shows that $\Nakr_C$ preserves $\Mod^C_{\fd}$. The proof is done.
\end{proof}

\subsection{Quasi-co-Frobenius coalgebras}

We investigate when $\Nakl_C$ and $\Nakr_C$ are equivalences.
To begin with, we consider the case where $C$ is co-Frobenius.
The definition of a co-Frobenius coalgebra was recalled in Subsection~\ref{subsec:special-classes-of-coalgebras}.
However, we use the following equivalent condition:
According to Iovanov \cite[Theorem 2.8]{MR2253657}, a coalgebra $C$ is co-Frobenius if and only if it has a {\em Frobenius pairing}, that is, a non-degenerate $C^*$-balanced bilinear map $C \times C \to \bfk$. The following definition is due to \cite[Section 6]{MR2078404}.

\begin{definition}
  Let $C$ be a co-Frobenius coalgebra with Frobenius pairing $\beta$. Then there is a unique linear map $\nu_C : C \to C$ determined by
  \begin{equation*}
    \beta(y, x) = \beta(\nu_C(x), y)
    \quad (x, y \in C),
  \end{equation*}
  which we call the {\em Nakayama automorphism} of $C$ with respect to $\beta$.
\end{definition}

It is known that $\nu_C$ is a coalgebra automorphism \cite[Lemma 3.18]{MR4560996}.
Given a coalgebra automorphism $f$ of $C$ and a right $C$-comodule $M$, we denote by $M^{(f)}$ the right $C$-comodule obtained by twisting the coaction by $f$.
The Nakayama functor for a co-Frobenius coalgebra is given as follows:

\begin{theorem}[{\cite[Lemma 3.20]{MR4560996}}]
  \label{thm:Nakayama-formula-co-Frobenius}
  For a co-Frobenius coalgebra $C$ with Nakayama automorphism $\nu$, there are natural isomorphisms
  \begin{equation*}
    \Nakl_C(M) \cong M^{(\nu^{-1})}
    \quad \text{and} \quad
    \Nakr_C(M) \cong M^{(\nu)}
    \quad (M \in \Mod^C).
  \end{equation*}
\end{theorem}
\begin{proof}
  Let $\beta : C \times C \to \bfk$ be the Frobenius pairing of $C$. Then the map
  \begin{equation*}
    \beta_{\ell} : C \to C^{*\rat},
    \quad c \mapsto \beta(c, -)
  \end{equation*}
  is bijective \cite[Section 1]{MR2078404}. Thus we obtain an isomorphism
  \begin{equation*}
    \Nakr_C(M)
    = C \otimes_{C^*} M
    \xrightarrow[\cong]{\quad \beta_{\ell} \otimes_{C^*} \id_M \quad}
    C^{*\rat} \otimes_{C^*} M
    \xrightarrow[\cong]{\quad \text{Lemma~\ref{lem:local-unit-2}} \quad} M
  \end{equation*}
  of vector spaces. The formula for $\Nakr_C(M)$ is obtained by transporting the coaction of $C$ through this isomorphism.
  The formula for $\Nakl_C(M)$ follows from that $\Nakl_C$ is right adjoint to $\Nakr_C$.
\end{proof}

In particular, $\Nakl_C$ and $\Nakr_C$ are equivalences if $C$ is co-Frobenius.
The converse does not holds: Let $C$ be a QcF coalgebra that is not co-Frobenius (see \cite{MR3125851} for a concrete example).
Then there is a co-Frobenius coalgebra $D$ that is Morita-Takeuchi equivalent to $C$ \cite{MR1904645}. By Lemma \ref{lem:Nakayama-MT-invariance} and Theorem \ref{thm:Nakayama-formula-co-Frobenius}, we see that $\Nakl_C$ and $\Nakr_C$ are equivalences, despite that $C$ is not co-Frobenius.

The above argument shows that $\Nakl_C$ and $\Nakr_C$ are equivalences if $C$ is QcF. With the help of Iovanov's result \cite[Lemma 1.7]{MR3150709} characterizing QcF coalgebras, one can prove that the converse holds. More specifically,

\begin{theorem}[{\it cf}. {\cite[Theorems 3.22 and 3.23]{MR4560996}}]
  \label{thm:Nakayama-QcF}
  For a coalgebra $C$, the following are equivalent:
  \begin{enumerate}
  \item $C$ is QcF.
  \item $\Nakl_C$ is an equivalence.
  \item $\Nakr_C$ is an equivalence.
  \item $\Nakr_C$ is exact and faithful.
  \end{enumerate}
  Suppose that these equivalent conditions are satisfied.
  Then we have
  \begin{equation}
    \label{eq:Nakayama-inj-hull-2}
    P(S) \cong E(\Nakl_C(S))
    \quad \text{and} \quad
    E(S) \cong P(\Nakr_C(S))
  \end{equation}
  for a simple right $C$-comodule $S$, where $E(S)$ and $P(S)$ are an injective hull and a projective cover of $S$, respectively.
\end{theorem}
\begin{proof}
  We have already proved (1) $\Rightarrow$ (2) and (1) $\Rightarrow$ (3).
  Since $\Nakr_C$ and $\Nakl_C$ are an adjoint pair, (2) and (3) are equivalent.
  It is trivial that (3) implies (4).
  Now we suppose that (4) holds and aim to show (1).
  Let $\{ S_i \}_{i \in I}$ be a complete set of representatives of isomorphism classes of simple right $C$-comodules.
  The right $C$-comodule $C$ is injective.
  By Lemma \ref{lem:Nakayama-C-star-rat-inj} and the assumption that $\Nakr_C$ is exact, $C^{*\rat}_{\ell}$ is also injective.
  Thus $C$ and $C^{*\rat}_{\ell}$ are decomposed as
  \begin{equation*}
    C \cong \bigoplus_{i \in I} E(S_i)^{\oplus a_i}
    \quad \text{and} \quad
    C^{*\rat}_{\ell} \cong \bigoplus_{i \in I} E(S_i)^{\oplus b_i}
  \end{equation*}
  for some cardinals $a_i$ and $b_i$.
  Since $\Hom^C(S_i, C) \cong S_i^* \ne 0$, we have $a_i > 0$ for all $i$.
  By \eqref{eq:Nakayama-dual} and the assumption that $\Nakr_C$ is faithful, we have
  \begin{equation*}
    \Hom^C(S_i, C^{*\rat}_{\ell})
    \cong \Nakr_C(S_i)^* \ne 0
  \end{equation*}
  for all $i$. This implies that $b_i > 0$ for all $i$.
  Now let $\kappa$ be an infinite cardinal bigger than all $a_i$'s and all $b_i$'s.
  Then we have $C^{\oplus \kappa} \cong (C^{*\rat}_{\ell})^{\oplus \kappa}$ by the above discussion.
  Thus, by \cite[Lemma 1.7]{MR3150709}, we conclude that $C$ is QcF, that is, (1) holds.
  We have proved that (1), (2), (3) and (4) are equivalent.
  The rest of the proof easily follows from Theorem \ref{thm:Nakayama-semiperfect}.
\end{proof}

Let $C$ be a QcF coalgebra, and let $\{ S_i \}_{i \in I}$ be a complete set of representatives of isomorphism classes of simple right $C$-comodules.
Then there is a permutation $\pi$ of the index set $I$ defined by $P(S_i) \cong E(S_{\pi(i)})$.
Following the terminology used in the representation theory of finite-dimensional algebras, we may call $\pi$ the {\em Nakayama permutation} of $C$.
According to \cite[Corollary 2.3]{MR3150709}, a QcF coalgebra is co-Frobenius if and only if its Nakayama permutation preserves the dimension of simple comodules.
Theorem \ref{thm:Nakayama-QcF} says that the Nakayama permutation is induced by the Nakayama functor $\Nakl_C$. Thus we have:

\begin{theorem}[{\cite[Theorem 3.24]{MR4560996}}]
  \label{thm:Nakayama-cF}
  For a QcF coalgebra $C$, the following are equivalent:
  \begin{enumerate}
  \item $C$ is co-Frobenius.
  \item $\dim_{\bfk} \Nakl_C(S) = \dim_{\bfk} S$ for all simple right $C$-comodule $S$.
  \item $\dim_{\bfk} \Nakr_C(S) = \dim_{\bfk} S$ for all simple right $C$-comodule $S$.
  \end{enumerate}
\end{theorem}

\subsection{Symmetric coalgebras}

A coalgebra $C$ is said to be {\em symmetric} \cite{MR2078404} if there is an injective homomorphism $C \to C^*$ of $C^*$-bimodules.
This class of coalgebras is characterized by the Nakayama functor as follows:

\begin{theorem}[{\cite[Theorem 3.27]{MR4560996}}]
  \label{thm:Nakayama-symm-coalg}
  For a coalgebra $C$, the following are equivalent:
  \begin{enumerate}
  \item $C$ is a symmetric coalgebra.
  \item $\Nakl_C$ is isomorphic to the identity functor.
  \item $\Nakr_C$ is isomorphic to the identity functor.
  \end{enumerate}
\end{theorem}
\begin{proof}
  A coalgebra automorphism $f$ of $C$ is said to be {\em coinner} if there is an invertible element $\alpha \in C^*$ such that $f(c) = \alpha \rightharpoonup c \leftharpoonup \alpha^{-1}$ for all $c \in C$.
  From a categorical point of view, a coalgebra automorphism $f$ of $C$ is coinner precisely if the autoequivalence $(-)^{(f)}$ on $\Mod^C$ induced by $f$ is isomorphic to the identity functor.
  It is known that a symmetric coalgebra is precisely a co-Frobenius coalgebra with coinner Nakayama automorphism \cite[Proposition 6.2]{MR2078404}. Thus (1) $\Rightarrow$ (2) and (1) $\Rightarrow$ (3) follow immediately from Theorem \ref{thm:Nakayama-formula-co-Frobenius}.

  It is obvious from Lemma \ref{lem:Nakayama-adjoint} that (2) and (3) are equivalent.
  To complete the proof, we prove (2) $\Rightarrow$ (1).
  We assume that (2) holds.
  Then, by Theorems \ref{thm:Nakayama-QcF} and \ref{thm:Nakayama-cF}, $C$ is co-Frobenius. Moreover, by Theorem \ref{thm:Nakayama-formula-co-Frobenius}, the Nakayama automorphism of $C$ is coinner. Thus $C$ is symmetric. The proof is done.
\end{proof}

\subsection{A Calabi-Yau type structure}

In some literature, a Calabi-Yau structure on a linear category $\mathcal{A}$ refers to a natural isomorphism
\begin{equation*}
  \Hom_{\mathcal{A}}(X, Y)^* \cong \Hom_{\mathcal{A}}(Y, X)
  \quad (X, Y \in \mathcal{A})
\end{equation*}
or an equivalent structure.
Here we point out that $\Proj(\Mod^C_{\fd})$ for a semiperfect coalgebra $C$ has a similar kind of structure:

\begin{theorem}
  \label{thm:Calabi-Yau-1}
  For a semiperfect coalgebra $C$, there is a natural isomorphism
  \begin{equation*}
    \Hom^C(P, M)^* \cong \Hom^C(M, \Nakr_C(P))
    \quad (P \in \Proj(\Mod^C_{\fd}), M \in \Mod^C_{\fd}).
  \end{equation*}
\end{theorem}
\begin{proof}
  We fix $P \in \Proj(\Mod^C_{\fd})$ and set $P^{\dagger} = {}_{C^*}\Hom(P, C^*)$.
  Since $P$ is projective as a left $C^*$-module \cite[Corollary 2.4.19]{MR1786197}, the map
  \begin{equation*}
    \theta_{P, M}: P^{\dagger} \otimes_{C^*} M \to {}_{C^*}\Hom(P, M),
    \quad \theta_{P,M}(f \otimes_{C^*} m)(p) = f(p) \cdot m
  \end{equation*}
  is a natural isomorphism for $M \in {}_{C^*}\Mod$ with the inverse written by a pair of dual bases of $P$ as a left $C^*$-module.
  There is a well-defined linear map
  \begin{equation*}
    \phi_P^{}: P^{\dagger} \otimes_{C^*} \Nak(P) \to \bfk,
    \quad \phi_P^{}(f \otimes_{C^*} c \otimes_{C^*} p) = \langle f(p), c \rangle,
  \end{equation*}
  where $\Nak = \Nakr_C$.
  We use this map to define the bilinear pairing
  \begin{equation*}
    \beta_{M, P}: \Hom^C(M, \Nak(P)) \times \Hom^C(P, M) \to \bfk,
    \quad \beta_M(f, g) = \phi_P^{} \theta_{P,\Nak(P)}^{-1}(f g)
  \end{equation*}
  for $M \in \Mod^C_{\fd}$.
  We note that every indecomposable projective right $C$-comodule is of the form $(C e)^*$ for some primitive idempotent $e \in C$.
  If $P = (C e)^*$ for some primitive idempotent $e$, then we have isomorphisms
  \begin{gather*}
    \Hom^C(M, \Nak(P)) \cong \Hom^C(M, C e) \cong (e M)^*, \\
    \Hom^C(P, M) \cong {}_{C^*}\Hom(C^* e, M) \cong e M
  \end{gather*}
  by \eqref{eq:idempotent-iso} and \eqref{eq:Nakayama-r-of-Ce-dual} and, through these isomorphisms, the pairing $\beta_{M,P}$ is identified with the canonical pairing between $e M$ and $(e M)^*$.
  Thus $\beta_{M,P}$ is non-degenerate in this case.
  One can also verify the non-degeneracy of $\beta_{M, P}$ in the general case by decomposing $P$ into the direct sum of indecomposable projective comodules.
  It is now routine to check that the linear map
  \begin{equation*}
    \alpha_{M, P}: \Hom^C(M, \Nak(P)) \to \Hom^C(P, M)^*, \quad f \mapsto \beta_{M, P}(f, -)
  \end{equation*}
  induced by the pairing $\beta_{M, P}$ meets the required condition.
\end{proof}

By Theorem \ref{thm:Nakayama-symm-coalg}, we have:

\begin{corollary}
  For a symmetric coalgebra $C$, there is a natural isomorphism
  \begin{equation*}
    \Hom^C(P, M)^* \cong \Hom^C(M, P)
    \quad (P \in \Proj(\Mod^C_{\fd}), M \in \Mod^C_{\fd}).
  \end{equation*}
\end{corollary}

\section{Categorical interpretation}
\label{sec:cat-interpr}

\subsection{The coHom functor for a linear category}

In this section, we rephrase results of the previous section in the language of category theory for the purpose of broader applications.
We first discuss what the coHom functor should be.
Let $\mathcal{A}$ be a linear category, and let $W$ be a vector space of dimension $\alpha$ (which is possibly infinite). The {\em copower} of $M \in \mathcal{A}$ by $W$, denoted by $W \otimes M$, is an object of $\mathcal{A}$ together with a natural isomorphism
\begin{equation*}
  \Hom_{\mathcal{A}}(W \otimes M, N) \cong \Hom_{\bfk}(W, \Hom_{\mathcal{A}}(M, N))
\end{equation*}
for $N \in \mathcal{A}$. By choosing a basis of $W$, we see that $W \otimes M$ exists if and only the direct sum $M^{\oplus \alpha} \in \mathcal{A}$ exists and, if they exist, they are isomorphic.

Now we fix an object $M \in \mathcal{A}$ and assume that the copower $W \otimes M$ exists for all objects $W \in \Vect$. Then the assignment $W \mapsto W \otimes M$ gives rise to a linear functor from $\Vect$ to $\mathcal{A}$. If a left adjoint of this functor exists, then we write it as $\coHom_{\mathcal{A}}(M, -)$. Thus we have a natural isomorphism
\begin{equation*}
  \Hom_{\mathcal{A}}(N, W \otimes M)
  \cong \Hom_{\bfk}(\coHom_{\mathcal{A}}(M, N), W)
\end{equation*}
for $N \in \mathcal{A}$.
In view of the case of coalgebras, we say that an object $M \in \mathcal{A}$ is {\em quasi-finite} if the functor $\coHom_{\mathcal{A}}(M, -)$ exists.
We denote by $\mathcal{A}_{\qf}$ the full subcategory of $\mathcal{A}$ consisting of quasi-finite objects. Then the {\em coHom functor}
\begin{equation*}
  \coHom_{\mathcal{A}} : (\mathcal{A}_{\qf})^{\op} \times \mathcal{A} \to \Vect
\end{equation*}
is defined.
If $\mathcal{A} = \Mod^C$ for some coalgebra $C$, then the functor $\coHom_{\mathcal{A}}$ is identical to the functor $\coHom^C$ introduced in Subsection~\ref{subsec:qf-comodules}.

Given a functor $F$, we denote by $F^{\ladj}$ and $F^{\radj}$ a left and a right adjoint of $F$, respectively, if they exist.
For later use, we remark the following relation between the coHom functor and adjunctions:

\begin{lemma}
  Let $\mathcal{A}$ and $\mathcal{B}$ be linear categories admitting arbitrary copowers, and let $F: \mathcal{A} \to \mathcal{B}$ be a linear functor. If $F$ has a left adjoint and a right adjoint, then there is a natural isomorphism
  \begin{equation}
    \label{eq:cohom-and-adjoints}
    \coHom_{\mathcal{B}}(F(M), N)
    \cong \coHom_{\mathcal{A}}(M, F^{\ladj}(N))
    \quad (M \in \mathcal{A}_{\qf}, N \in \mathcal{B}).
  \end{equation}
\end{lemma}
\begin{proof}
  We first note that $F$ preserves copowers. Indeed, we have
  \begin{align*}
    \Hom_{\mathcal{B}}(F(W \otimes M), N)
    & \cong \Hom_{\mathcal{A}}(W \otimes M, F^{\radj}(N)) \\
    & \cong \Hom_{\bfk}(W, \Hom_{\mathcal{A}}(M, F^{\radj}(N))) \\
    & \cong \Hom_{\bfk}(W, \Hom_{\mathcal{B}}(F(M), N)) \\
    & \cong \Hom_{\bfk}(W \otimes F(M), N)
  \end{align*}
  for all $W \in \Vect$, $M \in \mathcal{A}$ and $N \in \mathcal{B}$. The Yoneda lemma shows
  \begin{equation}
    \label{eq:copower-preserving}
    F(W \otimes M) \cong W \otimes F(M)
    \quad (W \in \Vect, M \in \mathcal{A}).
  \end{equation}
  For $W \in \Vect$, $M \in \mathcal{A}_{\qf}$ and $N \in \mathcal{B}$, we use the formula \eqref{eq:copower-preserving} to show:
  \begin{align*}
    \Hom_{\bfk}(\coHom_{\mathcal{B}}(F(M), N), W)
    & \cong \Hom_{\mathcal{B}}(N, W \otimes F(M)) \\
    & \cong \Hom_{\mathcal{B}}(N, F(W \otimes M)) \\
    & \cong \Hom_{\mathcal{A}}(F^{\ladj}(N), W \otimes M) \\
    & \cong \Hom_{\bfk}(\coHom_{\mathcal{B}}(M, F^{\ladj}(N)), W).
  \end{align*}
  The proof is completed by the Yoneda lemma.
\end{proof}

\subsection{The universal property of the Nakayama functor}
\label{subsec:Nakayama-uni-pro}

The Nakayama functor has a certain universal property.
To explain this, we first recall the notion of {\em coends} \cite{MR1712872}.
Let $\mathcal{A}$ and $\mathcal{V}$ be categories, and let $T: \mathcal{A}^{\op} \times \mathcal{A} \to \mathcal{V}$ be a functor. A {\em dinatural transformation} from $T$ to an object $V \in \mathcal{V}$ is a family
\begin{equation*}
  \{ i_X : T(X, X) \to V \}_{X \in \Obj(\mathcal{A})}
\end{equation*}
of morphisms in $\mathcal{V}$ such that the equation
\begin{equation*}
  i_X \circ T(f, \id_X) = i_Y \circ T(\id_Y, f)
\end{equation*}
holds for all morphisms $f: X \to Y$ in $\mathcal{V}$. A {\em coend} of $T$ is an object of $\mathcal{V}$ equipped with a `universal' dinatural transformation from $T$.
A coend of $T$ is customary denoted by using an integral symbol as
\begin{equation*}
  \int^{X \in \mathcal{A}} T(X, X).
\end{equation*}

We note the following relation between coends and adjunctions:

\begin{lemma}[{\cite[Lemma 3.9]{MR2869176}}]
  Let $T: \mathcal{B}^{\op} \times \mathcal{A} \to \mathcal{V}$ and $F: \mathcal{A} \to \mathcal{B}$ be functors, where $\mathcal{A}$, $\mathcal{B}$ and $\mathcal{V}$ are categories. If $F$ has a right adjoint, then we have
  \begin{equation}
    \label{eq:coend-and-adjoint}
    \int^{X \in \mathcal{A}} T(F(X), X)
    \cong \int^{Y \in \mathcal{B}} T(Y, F^{\radj}(Y))
  \end{equation}
  in the sense that the coend on the left-hand side exists if and only if the right one exists and, if they exist, they are canonically isomorphic.
\end{lemma}

A basic and important example of coends is found in Tannaka theory. Let $C$ be a coalgebra. The Tannaka reconstruction theorem states that the coalgebra $C$ is recovered from the category $\Mod^C_{\fd}$ as a coend as
\begin{equation}
  C = \int^{X \in \Mod^C_{\fd}} X^* \otimes_{\bfk} X,
\end{equation}
where the universal dinatural transformation for $C$ is given by
\begin{equation}
  \label{eq:Tannaka-1}
  i_X : X^* \otimes_{\bfk} X \to C,
  \quad i_X(x^* \otimes x) = \langle x^*, x_{(0)} \rangle x_{(1)}
\end{equation}
for $x^* \in X^*$ and $x \in X \in \Mod^C_{\fd}$.

We note that $i_X$ is in fact a homomorphism of $C$-bicomodules.
Let ${}^C\Mod^C$ denote the category of $C$-bimodules.
In the same way as we do in Tannaka theory, we can show that $C$ is in fact a coend of the functor
\begin{equation}
  \label{eq:Tannaka-2}
  (\Mod^C_{\fd})^{\op} \times \Mod^C_{\fd} \to {}^C\Mod^C,
  \quad (X, Y) \mapsto X^* \otimes_{\bfk} Y
\end{equation}
with the universal dinatural transformation~\eqref{eq:Tannaka-1}.

Now we prove that $\Nakr_C(M)$ is a coend of a certain functor.
The idea is simple:
We know that $C$ is a coend in ${}^C\Mod^C$.
If we apply a cocontinuous functor to the coend $C$, then the result is also a coend.

\begin{lemma}[{\cite[Lemma 2.11]{MR4560996}}]
  \label{lem:Nakayama-r-coend}
  Let $C$ be a coalgebra. Then we have
  \begin{equation*}
    \Nakr_C(M) = \int^{X \in \Mod^C_{\fd}}
    \coHom^C(X, M) \otimes_{\bfk} X
    \quad (M \in \Mod^C).
  \end{equation*}
\end{lemma}

Although we omit the details, we note that $\Nakl_C(M)$ is expressed as an end, that is, the dual notion of coends; see \cite[Theorem 3.6]{MR4560996}.

\begin{proof}
  We fix $M \in \Mod^C$ and consider the functor
  \begin{equation}
    \label{eq:Nakayama-r-coend-proof-1}
    {}^C\Mod^C \to \Mod^C,
    \quad P \mapsto P \otimes_{C^*} M.
  \end{equation}
  This functor has a right adjoint. Indeed, we have
  \begin{align*}
    \Hom^C(P \otimes_{C^*} M, N)
    & \cong {}_{C^*}\Hom_{C^*}(P, \Hom_{\bfk}(M, N)) \\
    & \cong {}^{C}\Hom^{C}(P, \Hom_{\bfk}(M, N)^{\rat})
  \end{align*}
  for $P \in {}^C\Mod^C$ and $N \in \Mod^C$ by the tensor-Hom adjunction, where $(-)^{\rat}$ means the rational part as $C$-bicomodules.
  Since $C \in {}^C\Mod^C$ is a coend of the functor \eqref{eq:Tannaka-2}, we see that $\Nakr_C(M)$ is a coend of the functor
  \begin{equation*}
    (\Mod^C_{\fd})^{\op} \times \Mod^C_{\fd} \to \Mod^C,
    \quad (X, Y) \mapsto (X^* \otimes_{\bfk} Y) \otimes_{C^*} M
  \end{equation*}
  by applying the functor~\eqref{eq:Nakayama-r-coend-proof-1} to $C$. The proof is completed by \eqref{eq:coHom-for-fd}.
\end{proof}

We give a proof of Lemma~\ref{lem:Nakayama-MT-invariance}, which we had skipped.

\begin{proof}[Proof of Lemma~\ref{lem:Nakayama-MT-invariance}]
  Let $\Phi : \Mod^C \to \Mod^D$ be an equivalence of linear categories, where $C$ and $D$ are coalgebras. Since a comodule is finite-dimensional if and only if it is of finite length, and since the length of an object is preserved by an equivalence, $\Phi$ induces an equivalence from $\Mod^C_{\fd}$ to $\Mod^D_{\fd}$.
  By the universal property given by Lemma \ref{lem:Nakayama-r-coend}, we have $\Phi \circ \Nakr_C \cong \Nakr_D \circ \Phi$. By taking right adjoints, we also have $\Nakl_C \circ \Phi^{-1} \cong \Phi^{-1} \circ \Nakl_D$. The proof is done.
\end{proof}

\subsection{Locally finite abelian categories and their ind-completions}

A {\em locally finite abelian category} \cite[\S1.8]{MR3242743} is a linear abelian category $\mathcal{A}$ such that every object of $\mathcal{A}$ is of finite length and $\Hom_{\mathcal{A}}(X, Y)$ is finite-dimensional for all objects $X, Y \in \mathcal{A}$.
By the result of Takeuchi \cite{MR472967} (see also \cite[\S1.8]{MR3242743}), a linear category is locally finite abelian if and only if it is equivalent to $\Mod^C_{\fd}$ for some coalgebra $C$.

Given a category $\mathcal{A}$, we denote by $\Ind(\mathcal{A})$ the ind-completion of $\mathcal{A}$ \cite{MR2182076}, that is, the category obtained from $\mathcal{A}$ by freely adjoining filtered colimits of objects of $\mathcal{A}$.
There is a canonical functor $\iota: \mathcal{A} \to \Ind(\mathcal{A})$ by which we regard $\mathcal{A}$ as a full subcategory of $\Ind(\mathcal{A})$.
If $\mathcal{A} = \Mod^C_{\fd}$ for some coalgebra $C$, then $\Ind(\mathcal{A})$ and $\iota$ are identified with $\Mod^C$ and the inclusion functor.
Keeping this fact in mind, we give characterizations of some classes of coalgebras by properties of their category of finite-dimensional comodules.
We first give a characterization of `one-sided' semiperfect coalgebras.

\begin{lemma}
  \label{lem:loc-fin-ab-semiperfect}
  For a locally finite abelian category $\mathcal{A}$, the following are equivalent:
  \begin{enumerate}
  \item $\mathcal{A} \approx \Mod_{\fd}^{C}$ for some right semiperfect coalgebra $C$.
  \item $\mathcal{A} \approx {}^C\Mod_{\fd}$ for some left semiperfect coalgebra $C$.
  \item Every object of $\mathcal{A}$ has a projective cover in $\mathcal{A}$.
  \item Every object of $\mathcal{A}$ has a projective cover in $\Ind(\mathcal{A})$.
  \item $\mathcal{A}$ has enough projective objects.
  \item $\Ind(\mathcal{A})$ has enough projective objects.
  \end{enumerate}
\end{lemma}
\begin{proof}
  It is trivial that (1) and (2) are equivalent.
  By the locally finiteness of $\mathcal{A}$, there is a coalgebra $C$ such that $\mathcal{A} \approx \Mod^C_{\fd}$. We fix such a coalgebra $C$ and identify $\mathcal{A}$ with $\Mod^C_{\fd}$.
  Then the equivalence between (1), (4), (5) and (6) is obtained by rephrasing known characterizations of semiperfect coalgebras \cite[Corollary 2.4.21 and Theorem 3.2.3]{MR1786197}.
  The implication (3) $\Rightarrow$ (4) is obvious.

  To complete the proof, it remains to show (1) $\Rightarrow$ (3). We assume that $C$ is right semiperfect.
  Let $S$ be a simple right $C$-comodule.
  Then an injective hull of $E(S^*) \in {}^C\Mod$ is finite-dimensional and $E(S^*)^* \in \Mod^C$ is a projective cover of $S$ (see the proof of \cite[Theorem 3.2.3]{MR1786197}).
  Thus (3) holds.
  The proof is done.
\end{proof}

Let $C$ be a coalgebra such that $\mathcal{A} \approx \Mod^C_{\fd}$.
Then the duality establishes equivalences ${}^C\Mod_{\fd} \approx (\Mod^C_{\fd})^{\op} \approx \mathcal{A}^{\op}$.
By applying Lemma~\ref{lem:loc-fin-ab-semiperfect} to $\mathcal{A}^{\op}$, we have:

\begin{lemma}
  \label{lem:loc-fin-ab-semiperfect-2}
  For a locally finite abelian category $\mathcal{A}$, the following are equivalent:
  \begin{enumerate}
  \item $\mathcal{A} \approx \Mod_{\fd}^{C}$ for some left semiperfect coalgebra $C$.
  \item $\mathcal{A} \approx {}^C\Mod_{\fd}$ for some right semiperfect coalgebra $C$.
  \item Every object of $\mathcal{A}$ has an injective hull in $\mathcal{A}$.
  \item $\mathcal{A}$ has enough injective objects.
  \end{enumerate}
\end{lemma}

By Lemmas \ref{lem:loc-fin-ab-semiperfect} and \ref{lem:loc-fin-ab-semiperfect-2}, semiperfect coalgebras are characterized as follows:

\begin{lemma}
  \label{lem:loc-fin-ab-bisemiperfect}
  For a locally finite abelian category $\mathcal{A}$, the following are equivalent:
  \begin{enumerate}
  \item $\mathcal{A} \approx \Mod_{\fd}^{C}$ for some semiperfect coalgebra $C$.
  \item Every object of $\mathcal{A}$ has a projective cover and an injective hull in $\mathcal{A}$.
  \item $\mathcal{A}$ has enough projective objects and enough injective objects.
  \end{enumerate}
\end{lemma}

QcF coalgebras are characterized as follows:

\begin{lemma}
  \label{lem:loc-fin-ab-QcF}
  For a locally finite abelian category $\mathcal{A}$, the following are equivalent:  \begin{enumerate}
  \item $\mathcal{A} \approx \Mod^Q_{\fd}$ for some QcF coalgebra $Q$.
  \item $\mathcal{A} \approx \Mod^Q_{\fd}$ for some semiperfect coalgebra $Q$ and $\Proj(\mathcal{A}) = \Ind(\mathcal{A})$.
  \end{enumerate}
\end{lemma}
\begin{proof}
  This result is noted at the end of Section 3.1 of \cite{MR3150709}.
  We explain in a bit more detail.
  First, let $Q$ be a QcF coalgebra and set $\mathcal{A} = \Mod^Q_{\fd}$.
  Then $Q$ is semiperfect by \cite[Corollary 3.3.6]{MR1786197}.
  Since every injective right $Q$-comodule is projective \cite[Theorem 3.3.4]{MR1786197}, we have $\Inj(\mathcal{A}) \subset \Proj(\mathcal{A})$.
  By applying the same argument to $Q^{\cop}$, we also have an inclusion
  \begin{equation*}
    \Proj(\mathcal{A}) = \Inj(\mathcal{A}^{\op}) \subset \Proj(\mathcal{A}^{\op}) = \Inj(\mathcal{A})
  \end{equation*}
  of classes of objects. Thus we have proved (1) $\Rightarrow$ (2).

  To prove the converse, we let $Q$ be a semiperfect coalgebra such that $\Proj(\mathcal{A}) = \Ind(\mathcal{A})$, where $\mathcal{A} := \Mod^Q_{\fd}$.
  We decompose $Q$ into a direct sum of indecomposable injective right $Q$-comodules, as $Q = \bigoplus_{i \in I} E_i$. The right semiperfectness of $Q$ implies that each $E_i$ is finite-dimensional \cite[Theorem 3.2.3]{MR1786197}.
  Thus $Q$ is in fact a direct sum of projective objects.
  This shows that $Q$ is projective, and therefore $Q$ is left QcF \cite[Theorem 3.3.4]{MR1786197}.
  By applying the same argument to $Q^{\cop}$, we see that $Q$ is also right QcF.
  The proof is done.
\end{proof}

\subsection{Nakayama functor for locally finite abelian categories}

Our categorical dictionary is now ready for translating the results on the Nakayama functor established in the previous section.
In view of Lemma \ref{lem:Nakayama-r-coend}, we introduce:

\begin{definition}
  \label{def:Nakayama-locally-finite}
  For a locally finite abelian category $\mathcal{A}$, we define the functor
  \begin{equation*}
    \Nakr_{\Ind(\mathcal{A})}: \Ind(\mathcal{A}) \to \Ind(\mathcal{A}),
    \quad M \mapsto \int^{X \in \mathcal{A}} \coHom_{\Ind(\mathcal{A})}(X, M) \otimes X
  \end{equation*}
  and call $\Nakr_{\Ind(\mathcal{A})}$ the {\em right exact Nakayama functor} for $\mathcal{A}$. We also define the left exact Nakayama functor $\Nakl_{\Ind(\mathcal{A})}$ as a right adjoint of $\Nakr_{\Ind(\mathcal{A})}$. 
\end{definition}

Lemmas~\ref{lem:Nakayama-adjoint} and \ref{lem:Nakayama-r-coend} imply that $\Nakr_{\Ind(\mathcal{A})}$ and $\Nakr_{\Ind(\mathcal{A})}$ are identified with $\Nakr_C$ and $\Nakl_C$, respectively, when $\mathcal{A} = \Mod^C_{\fd}$ for a coalgebra $C$ and $\Ind(\mathcal{A})$ is identified with $\Mod^C$.
Thus the functors $\Nakr_{\Ind(\mathcal{A})}$ and $\Nakr_{\Ind(\mathcal{A})}$ are defined for any locally finite abelian category $\mathcal{A}$.
Theorem \ref{thm:Nakayama-semiperfect} is translated by Lemma \ref{lem:loc-fin-ab-bisemiperfect} as follows:

\begin{theorem}
  \label{thm:locally-fin-ab-semiperfect}
  Let $\mathcal{A}$ be a locally finite abelian category with enough projective objects and enough injective objects.
  Then the functor $\Nakr_{\Ind(\mathcal{A})}$ restricts to an endofunctor on $\mathcal{A}$, which we denote by $\Nakr_{\mathcal{A}} : \mathcal{A} \to \mathcal{A}$.
  The functor $\Nakr_{\mathcal{A}}$ induces an equivalence $\Nakr_{\mathcal{A}}: \Proj(\mathcal{A}) \to \Inj(\mathcal{A})$.
\end{theorem}

Based on Theorem \ref{thm:Nakayama-QcF}, we prove:

\begin{theorem}
  \label{thm:locally-fin-ab-QcF}
  For a locally finite abelian category $\mathcal{A}$, the following are equivalent:
  \begin{enumerate}
  \item $\mathcal{A} \approx \Mod^Q_{\fd}$ for some QcF coalgebra $Q$.
  \item $\Nakr_{\Ind(\mathcal{A})}$ is an auto-equivalence of $\Ind(A)$.
  \item $\Nakr_{\Ind(\mathcal{A})}$ is exact and faithful.
  \item $\Nakr_{\Ind(\mathcal{A})}$ induces an auto-equivalence of $\mathcal{A}$.
  \end{enumerate}
\end{theorem}
\begin{proof}
  Theorem \ref{thm:Nakayama-QcF} states that (1), (2) and (3) are equivalent.
  It is trivial that (2) implies (4).
  To complete the proof, we shall show that (4) implies (3).
  We assume that $\Nakr_{\Ind(\mathcal{A})}$ induces an endofunctor on $\mathcal{A}$, say $\Nakr_{\mathcal{A}}: \mathcal{A} \to \mathcal{A}$.
  Then the endofunctor $\Ind(\Nakr_{\mathcal{A}})$ on $\Ind(\mathcal{A})$ induced by $\Nakr_{\mathcal{A}}$ is isomorphic to $\Nakr_{\Ind(\mathcal{A})}$.
  Thus, if $\Nakr_{\mathcal{A}}$ is an equivalence, so is the functor $\Ind(\Nakr_{\mathcal{A}})$. The proof is done.
\end{proof}

Theorem \ref{thm:Nakayama-symm-coalg} implies:

\begin{theorem}
  \label{thm:locally-fin-ab-symmetric}
  For a locally finite abelian category $\mathcal{A}$, the following are equivalent:
  \begin{enumerate}
  \item $\mathcal{A} \approx \Mod^Q_{\fd}$ for some symmetric coalgebra $Q$.
  \item $\Nakr_{\Ind(\mathcal{A})}$ is isomorphic to the identity functor on $\Ind(\mathcal{A})$.
  \item The restriction of $\Nakr_{\Ind(\mathcal{A})}$ to $\mathcal{A}$ is the identity functor in $\mathcal{A}$.
  \end{enumerate}
\end{theorem}

\subsection{Interaction with double adjoints}

In the next section, we give some applications of the Nakayama functor to tensor categories.
The following theorem is essential ({\it cf}. Observation~\ref{obs:intro-2}):

\begin{theorem}
  \label{thm:Nakayama-double-adj}
  Let $\mathcal{A}$ and $\mathcal{B}$ be locally finite abelian categories, and let $F: \Ind(\mathcal{A}) \to \Ind(\mathcal{B})$ be a cocontinuous linear functor such that the double left adjoint $F^{\lladj} = (F^{\ladj})^{\ladj}$ exists.
  Then there is an isomorphism
  \begin{equation}
    \label{ew:Nakayama-double-adj}
    F \circ \Nakr_{\Ind(\mathcal{A})} \cong \Nakr_{\Ind(\mathcal{B})} \circ F^{\lladj}.
  \end{equation}
\end{theorem}
\begin{proof}
  Since $\Mod^C$ for a coalgebra $C$ is locally presentable, $F$ has a right adjoint.
  For $M \in \Ind(\mathcal{A})$, we have natural isomorphisms
  \begin{align*}
    F(\Nakr_{\Ind(\mathcal{A})}(M))
    & \cong \textstyle \int_{X \in \mathcal{A}} F(\coHom_{\Ind(\mathcal{A})}(X, M) \otimes X) \\
    & \cong \textstyle \int_{X \in \mathcal{A}} \coHom_{\Ind(\mathcal{A})}(X, M) \otimes F(X) \\
    & \cong \textstyle \int_{Y \in \mathcal{B}} \coHom_{\Ind(\mathcal{A})}(F^{\ladj}(Y), M) \otimes Y \\
    & \cong \textstyle \int_{Y \in \mathcal{B}} \coHom_{\Ind(\mathcal{B})}(Y, F^{\lladj}(M)) \otimes Y \\
    & = \Nakr_{\Ind(\mathcal{B})}(F^{\lladj}(M)),
  \end{align*}
  where the first isomorphism follows from the cocontinuity of $F$,
  the second from \eqref{eq:copower-preserving},
  the third from \eqref{eq:coend-and-adjoint},
  and the fourth from \eqref{eq:cohom-and-adjoints}.
\end{proof}

\section{Applications to tensor categories}
\label{sec:tensor-cat}

\subsection{Frobenius tensor categories}

In this section, we give some applications of our results to tensor categories.
We mostly follow \cite{MR3242743} for notation and terminology on monoidal categories.
In view of Mac Lane's strictness theorem, we assume that all monoidal category are strict.
Given a monoidal category $\mathcal{C}$, we usually denote by $\otimes$ and $\unitobj$ the monoidal product and the unit object of $\mathcal{C}$, respectively.
The category $\mathcal{C}^{\rev}$ is the monoidal category obtained from $\mathcal{C}$ by reversing the order of the monoidal product.

Following \cite{MR3242743}, a left dual object of $X \in \mathcal{C}$ is an object $X^{\vee} \in \mathcal{C}$ equipped with morphisms $\eval_X : X^{\vee} \otimes X \to \unitobj$ and $\coev_X : \unitobj \to X \otimes X^{\vee}$ satisfying certain equations.
A right dual object of $X$ is written as ${}^{\vee}\!X$.
We recall that $\mathcal{C}$ is left (right) rigid if every object of $\mathcal{C}$ has a left (right) dual object. A rigid monoidal category is a left and right rigid monoidal category.

\begin{definition}
  A {\em tensor category} \cite{MR3242743} is a locally finite abelian category $\mathcal{C}$ equipped with a structure of a rigid monoidal category such that the tensor product $\otimes: \mathcal{C} \times \mathcal{C} \to \mathcal{C}$ is $\bfk$-linear in each variable and the unit object $\unitobj \in \mathcal{C}$ is a simple object whose endomorphism algebra is isomorphic to $\bfk$.
\end{definition}

Our main interest in this section is the following class of tensor categories:

\begin{definition}
  A {\em Frobenius tensor category} \cite{MR3410615} is a tensor category $\mathcal{C}$ such that every simple object of $\mathcal{C}$ has an injective hull in $\mathcal{C}$.
\end{definition}

\subsection{Equivalent conditions for Frobenius property}

Let $H$ be a Hopf algebra.
Then the locally finite abelian category $\Mod^{H}_{\fd}$ has a natural structure of a left rigid monoidal category, which may not be a tensor category.
Indeed, $\Mod^H_{\fd}$ is rigid if and only if the antipode of $H$ is bijective \cite{MR1098991} and there are some examples of Hopf algebras with non-bijective antipodes \cite{MR292876}.
Thus, from a viewpoint of general theory of Hopf algebras, the following situation could be important:

\begin{assumption}
  \label{assumption:pretensor-cat}
  $\mathcal{C}$ is a locally finite abelian category equipped with a structure of a left rigid monoidal category such that the tensor product of $\mathcal{C}$ is $\bfk$-linear and exact in each variable, and the unit object $\unitobj$ of $\mathcal{C}$ is a simple object whose endomorphism algebra is isomorphic to $\bfk$.
\end{assumption}

If $\mathcal{C}$ satisfies Assumption \ref{assumption:pretensor-cat}, then $\Ind(\mathcal{C})$ has a natural monoidal structure extending that of $\mathcal{C}$.
The exactness of the tensor product of $\mathcal{C}$ implies that the tensor product of $\Ind(\mathcal{C})$ preserves colimits.
Since the tensor product of $\mathcal{C}$ is faithful by \cite[Exercise 4.3.11]{MR3242743}, so is the tensor product of $\Ind(\mathcal{C})$.
Now we prove the following formula of the Nakayama functor:

\begin{lemma}
  \label{lem:Nakayama-tensor-cat-formula}
  Let $\mathcal{C}$ satisfy Assumption \ref{assumption:pretensor-cat}.
  Then we have
  \begin{gather}
    \label{eq:Nakayama-tensor-cat-formula}
    \Nakr_{\Ind(\mathcal{C})}(X \otimes Y)
    \cong \Nakr_{\Ind(\mathcal{C})}(X) \otimes \mathbb{D}(Y), \\
    \label{eq:Nakayama-tensor-cat-formula-2}
    \Nakl_{\Ind(\mathcal{C})}(X \otimes Y)
    \cong \mathbb{D}(X) \otimes \Nakl_{\Ind(\mathcal{C})}(Y)
  \end{gather}
  for $X, Y \in \Ind(\mathcal{C})$, where $\mathbb{D} : \Ind(\mathcal{C}) \to \Ind(\mathcal{C})$ is the functor induced by the double left dual functor on the category $\mathcal{C}$.
  In particular, we have
  \begin{equation}
    \label{eq:Nakayama-tensor-cat-formula-3}
    \Nakr_{\Ind(\mathcal{C})}(X) \cong \modobj \otimes \mathbb{D}(X), \quad
    \Nakl_{\Ind(\mathcal{C})}(X) \cong \mathbb{D}(X) \otimes \overline{\modobj}
    \quad (X \in \Ind(\mathcal{C})),
  \end{equation}
  where $\modobj = \Nakr_{\mathcal{C}}(\unitobj)$ and $\overline{\modobj} = \Nakl_{\mathcal{C}}(\unitobj)$.
\end{lemma}
\begin{proof}
  We write $\Nak = \Nakr_{\Ind(\mathcal{C})}$ for brevity.
  We fix an object $Y \in \mathcal{C}$.
  By applying Theorem \ref{thm:Nakayama-double-adj} to $F = (-) \otimes Y^{\vee\vee}$, we have a natural isomorphism
  \begin{equation*}
    \Nak(X \otimes Y)
    = \Nak F^{\lladj}(X)
    \cong F(\Nak(X)) = \Nak(X) \otimes Y^{\vee\vee}
    \quad (X \in \Ind(\mathcal{C})).
  \end{equation*}
  Thus \eqref{eq:Nakayama-tensor-cat-formula} holds for all $X \in \Ind(\mathcal{C})$ and $Y \in \mathcal{C}$.
  By the cocontinuity of $\Nak$ and $\otimes$, the formula \eqref{eq:Nakayama-tensor-cat-formula} actually holds for all $X, Y \in \Ind(\mathcal{C})$. The isomorphism \eqref{eq:Nakayama-tensor-cat-formula-2} is established by the universal property of $\Nakl_{\Ind(\mathcal{C})}$ as an end (see \cite{MR4560996}).
\end{proof}

Let $H$ be a Hopf algebra.
A left (right) integral on $H$ is a morphism $H \to \bfk$ of left (right) $H$-comodules, where $\bfk$ is the trivial comodule.
It is known that $H$ is co-Frobenius if and only if $H$ has a non-zero (left or right) integral on $H$,  if and only if $H^{*\rat}_{\ell} \ne 0$
\cite[Corollary 5.2.4 and Theorem 5.3.2]{MR1786197}.
Moreover, the antipode of a co-Frobenius Hopf algebra is bijective \cite[Corollary 5.4.6]{MR1786197}.
By Lemma \ref{lem:Nakayama-C-star-rat-zero}, the condition $H^{*\rat}_{\ell} = 0$ is also equivalent to $\Nakr_H \ne 0$.
As a categorical generalization of these results, we prove:

\begin{theorem}[{\it cf}. {\cite[Theorem 5.2]{MR4560996}} and {\cite[Lemma 4.4]{2023arXiv230314687S}}]
  \label{thm:one-sided-rigidity}
  Let $\mathcal{C}$ satisfy Assumption \ref{assumption:pretensor-cat}, and let $Q$ be a coalgebra such that $\mathcal{C} \approx \Mod^Q_{\fd}$ as linear categories.
  Then the following are equivalent:
  \begin{enumerate}
  \item $Q$ is QcF.
  \item $Q$ is right QcF.
  \item $Q$ is right semiperfect.
  \item $Q^{*\rat}_{\ell} \ne 0$.
  \item $Q$ is left QcF.
  \item $Q$ is left semiperfect.
  \item $Q^{*\rat}_{r} \ne 0$.
  \item There is a non-zero object of $\mathcal{C}$ which is projective or injective.
  \item $\Hom_{\Ind(\mathcal{C})}(E, X) \ne 0$ for some $X \in \mathcal{C}$ and injective $E \in \Ind(\mathcal{C})$.
  \item The functor $\Nak := \Nakr_{\Ind(\mathcal{C})}$ is an auto-equivalence of $\Ind(\mathcal{C})$.
  \item The functor $\Nak$ induces an auto-equivalence of $\mathcal{C}$.
  \item The functor $\Nak$ is not zero.
  \end{enumerate}
  Suppose that these equivalent conditions are satisfied.
  Then $\mathcal{C}$ is rigid.
  Thus $\mathcal{C}$ is a Frobenius tensor category.
\end{theorem}
\begin{proof}
  (4) and (12) are equivalent by Lemma \ref{lem:Nakayama-C-star-rat-zero}.
  (1), (10) and (11) are equivalent by Theorem \ref{thm:locally-fin-ab-QcF}.
  The general theory of coalgebras \cite[Chapter 3]{MR1786197} implies
  (1) $\Rightarrow$ (2), (2) $\Rightarrow$ (3), (3) $\Rightarrow$ (4),
  (1) $\Rightarrow$ (5), (5) $\Rightarrow$ (6), (6) $\Rightarrow$ (7)
  and (1) $\Rightarrow$ (8).
  To show that all the conditions (1)--(12) are equivalent, it suffices to show (12) $\Rightarrow$ (1), (8) $\Rightarrow$ (9), (9) $\Rightarrow$ (7) and (7) $\Rightarrow$ (1).

  \medskip \underline{(12) $\Rightarrow$ (1)}.
  This part was proved in \cite[Lemma 4.4]{2023arXiv230314687S} under the assumption that $\mathcal{C}$ is a tensor category.
  We give a proof based on Lemma~\ref{thm:locally-fin-ab-QcF}.
  Let $\mathcal{C}$ satisfy Assumption \ref{assumption:pretensor-cat} and assume $\Nak \ne 0$.
  Then the functor $\Nakl := \Nakl_{\Ind(\mathcal{C})}$ is non-zero as well.
  Hence, by \eqref{eq:Nakayama-tensor-cat-formula-3}, the object  $\overline{\modobj} = \Nakl(\unitobj)$ is non-zero.

  We first show that the endofunctor $\mathbb{D}: \Ind(\mathcal{C}) \to \Ind(\mathcal{C})$ of Lemma \ref{lem:Nakayama-tensor-cat-formula} is left exact.
  Let $0 \to X \xrightarrow{\ p \ } Y \xrightarrow{\ q \ } Z$ be an exact sequence in $\Ind(\mathcal{C})$.
  By the left exactness of $\Nakl$ and \eqref{eq:Nakayama-tensor-cat-formula-3}, we have an exact sequence
  \begin{equation*}
    0 \to \mathbb{D}(X) \otimes \overline{\modobj}
    \xrightarrow{\quad \mathbb{D}(p) \otimes \id \quad}
    \mathbb{D}(Y) \otimes \overline{\modobj}
    \xrightarrow{\quad \mathbb{D}(q) \otimes \id \quad}
    \mathbb{D}(Z) \otimes \overline{\modobj}
  \end{equation*}
  in $\Ind(\mathcal{C})$. By the exactness of the functor $(-) \otimes \overline{\modobj}$, we have
  \begin{equation*}
    \Ker(\mathbb{D}(p)) \otimes \overline{\modobj} = 0
    \quad \text{and} \quad
    \frac{\Ker(\mathbb{D}(q))}{\Img(\mathbb{D}(p))} \otimes \overline{\modobj} = 0.
  \end{equation*}
  Since the functor $(-) \otimes \overline{\modobj}$ is also faithful, we have $\Ker(\mathbb{D}(p)) = 0$ and $\Ker(\mathbb{D}(q)) = \Img(\mathbb{D}(p))$.
  We have proved that $\mathbb{D}$ is left exact.

  The functor $\Nak$ is not only right exact, but also left exact by \eqref{eq:Nakayama-tensor-cat-formula-3} and the left exactness of $\mathbb{D}$.
  Since the endofunctor $X \mapsto X^{\vee\vee}$ on $\mathcal{C}$ is faithful, so is $\mathbb{D}$.
  Thus, by \eqref{eq:Nakayama-tensor-cat-formula-3}, $\Nak$ is faithful.
  Now, by Lemma~\ref{thm:locally-fin-ab-QcF}, we conclude that $Q$ is QcF.

  \medskip
  \underline{(8) $\Rightarrow$ (9)}.
  It is known that an object of a tensor category is projective if and only if it is injective \cite[Proposition 6.1.3]{MR3242743}.
  Although the rigidity is used in the proof presented in \cite{MR3242743}, one can prove that the same holds for categories satisfying Assumption~\ref{assumption:pretensor-cat} in the basically same manner (see \cite[\S5.2]{MR4560996}).
  Hence (8) implies that $\mathcal{C}$ has a non-zero injective object.
  Now it is trivial that (9) holds.
  
  \medskip
  \underline{(9) $\Rightarrow$ (7)}.
  We assume that there are an object $X \in \mathcal{C}$, an injective object $E \in \Ind(\mathcal{C})$ and a non-zero morphism $E \to X$ in $\Ind(\mathcal{C})$.
  For simplicity, we identify $\mathcal{C}$ and $\Ind(\mathcal{C})$ with $\Mod^Q_{\fd}$ and $\Mod^Q$, respectively.
  By replacing $E$ with its direct summand, we may assume that $E$ is indecomposable.
  Then $E$ is a direct summand of $Q$. Hence we have a non-zero morphism $Q \to X$ of right $Q$-comodules. Since $X$ is finite-dimensional, we have isomorphisms
  \begin{equation*}
    {}^Q\Hom(X^*, Q^{*\rat}_{r})
    \cong \Hom_{Q^*}(X^*, Q^{*})
    \cong \Hom^Q(Q, X) \ne 0,
  \end{equation*}
  which implies (7).

  \medskip
  \underline{(7) $\Rightarrow$ (1)}.
  We have already proved that (1) and (4) of this theorem are equivalent.
  We note that $\mathcal{C}^{\op,\rev}$ also satisfies Assumption~\ref{assumption:pretensor-cat} and there is an equivalence $\mathcal{C}^{\op,\rev} \approx {}^Q\Mod_{\fd}$ of linear categories.
  The equivalence (1) $\Leftrightarrow$ (7) is proved by applying (1) $\Leftrightarrow$ (4) to $\mathcal{C}^{\op,\rev}$.

  \medskip
  \underline{Rigidity of $\mathcal{C}$}.
  We assume that (1)--(12) are satisfied.
  By Theorem \ref{thm:locally-fin-ab-QcF}, the right exact Nakayama functor for $\mathcal{C}$ and its right adjoint induce autoequivalences on $\mathcal{C}$, which we denote by $\Nakr_{\mathcal{C}}$ and $\Nakl_{\mathcal{C}}$, respectively.
  The functors $\Nakr_{\mathcal{C}}$ and $\Nakl_{\mathcal{C}}$ are mutually quasi-inverse.
  Now we set $\modobj = \Nakr_{\mathcal{C}}(\unitobj)$ and $\overline{\modobj} = \Nakl_{\mathcal{C}}(\unitobj)$.
  By \eqref{eq:Nakayama-tensor-cat-formula-3} and $\Nakl_{\mathcal{C}} \Nakr_{\mathcal{C}} \cong \id_{\mathcal{C}} \cong \Nakr_{\mathcal{C}} \Nakl_{\mathcal{C}}$, we have isomorphisms $\modobj \otimes \overline{\modobj} \cong \unitobj \cong \overline{\modobj} \otimes \modobj$ (this part is slight technical, though; see \cite{MR4560996}). From this, we conclude that the double dual functor is an auto-equivalence on $\mathcal{C}$.
  Hence $\mathcal{C}$ is rigid. The proof is done.
\end{proof}

\subsection{A categorical analogue of the modular function}

From now on, we introduce interesting results on Frobenius tensor categories from \cite{2023arXiv230314687S,MR4560996}.
Let $\mathcal{C}$ be a Frobenius tensor category.
By Theorems \ref{thm:Nakayama-QcF} and \ref{thm:one-sided-rigidity}, the Nakayama functor $\Nakr_{\Ind(\mathcal{C})}$ restricts to an autoequivalence $\Nakr_{\mathcal{C}}$ on $\mathcal{C}$. Now we define
\begin{equation}
  \modobj_{\mathcal{C}} := \Nakr_{\mathcal{C}}(\unitobj) \in \mathcal{C}.
\end{equation}
The object $\modobj_{\mathcal{C}}$ corresponds to the modular function of a co-Frobenius Hopf algebra (see \cite{MR4560996}).
According to the nomenclature in the Hopf algebra theory, we introduce the following terminology:

\begin{definition}
  We call $\modobj_{\mathcal{C}}$ the {\em modular object} of $\mathcal{C}$.
  We say that $\mathcal{C}$ is {\em unimodular} if $\modobj_{\mathcal{C}} \cong \unitobj$.
\end{definition}

By the same way as Lemma \ref{lem:Nakayama-tensor-cat-formula}, we have:

\begin{theorem}[{\it cf}. {\cite[Theorems 4.4 and 4.5]{MR4042867}}]
  \label{thm:Frob-tensor-cat-Nakayama}
  Let $\mathcal{C}$ be a Frobenius tensor category.
  Then there are natural isomorphisms
  \begin{equation*}
    {}^{\vee\vee}\!X \otimes \Nakr_{\mathcal{C}}(Y)
    \cong \Nakr_{\mathcal{C}}(X \otimes Y)
    \cong \Nakr_{\mathcal{C}}(X) \otimes Y^{\vee\vee}
  \end{equation*}
  for $X, Y \in \mathcal{C}$. In particular,
  \begin{equation}
    \label{eq:Nakayama-Frob-tensor-cat-formula}
    {}^{\vee\vee\!}X \otimes \modobj_{\mathcal{C}}
    \cong \Nakr_{\mathcal{C}}(X)
    \cong \modobj_{\mathcal{C}} \otimes X^{\vee\vee}
    \quad (X \in \mathcal{C}).
  \end{equation}
\end{theorem}

We give some immediate corollaries of this theorem.
Let $\mathcal{C}$ be a Frobenius tensor category.
Since $\Nakr_{\mathcal{C}}$ is an auto-equivalence of $\mathcal{C}$, we have:

\begin{corollary}
  The object $\modobj_{\mathcal{C}}$ is invertible, that is,
  \begin{equation*}
    \modobj_{\mathcal{C}} \otimes \modobj_{\mathcal{C}}^{\vee}
    \cong \unitobj \cong
    \modobj_{\mathcal{C}}^{\vee} \otimes \modobj_{\mathcal{C}}.
  \end{equation*}
\end{corollary}

Corollary~\ref{cor:ACE15-Rem-2-10} below has been proved in \cite[Remark 2.10]{MR3410615} under the assumption that $\mathcal{C}$ admits a dimension function.
Given a simple object $S$, we denote by $E(S)$ and $P(S)$ an injective hull and a projective cover of $S$, respectively.
By Theorems~\ref{thm:Nakayama-QcF} and \ref{thm:Frob-tensor-cat-Nakayama}, we have:

\begin{corollary}[{\cite[Corollary 5.10]{MR4560996}}]
  \label{cor:ACE15-Rem-2-10}
  There are isomorphisms
  \begin{equation*}
    E(S) \cong P(\modobj_{\mathcal{C}} \otimes S^{\vee\vee}),
    \quad
    P(S) \cong E(\modobj_{\mathcal{C}}^{\vee} \otimes {}^{\vee\vee}S)
  \end{equation*}
  for all simple objects $S \in \mathcal{C}$. In particular, $\modobj_{\mathcal{C}}$ is the top of $E(\unitobj)$ as well as the dual of the socle of $P(\unitobj)$.
\end{corollary}

By Theorems~\ref{thm:Nakayama-symm-coalg} and \ref{thm:Frob-tensor-cat-Nakayama}, we have:

\begin{corollary}[{\cite[Corollary 5.11]{MR4560996}}]
  \label{cor:Frob-tensor-cat-symmetricity}
  For a Frobenius tensor category $\mathcal{C}$, the following are equivalent:
  \begin{enumerate}
  \item $\mathcal{C}$ is unimodular and the double dual functor $(-)^{\vee\vee}$ on $\mathcal{C}$ is isomorphic to the identity functor.
  \item $\mathcal{C} \approx \fdmod^Q$ for some symmetric coalgebra $Q$.
  \end{enumerate}
\end{corollary}

\subsection{A categorical analogue of Radford's $S^4$-formula}

Radford's $S^4$-formula expresses the fourth power of the antipode of a finite-dimensional Hopf algebra in terms of certain elements defined by using integrals.
This formula has been generalized to co-Frobenius Hopf algebras in \cite[Theorem 2.8]{MR2278058} and to finite tensor categories in \cite{MR2097289}.
It was pointed out that Radford's $S^4$-formula for finite tensor categories is obtained from a general result on the Nakayama functor for finite abelian categories in \cite{MR4042867}.
As we have established the theory of Nakayama functors for locally finite abelian categories, we obtain the following theorem in the same way as \cite{MR4042867}.

\begin{theorem}
  \label{thm:Radford-S4-Frob-ten-cat}
  For a Frobenius tensor category $\mathcal{C}$, there is an isomorphism
  \begin{equation}
    \label{eq:Frob-ten-cat-Radford-S4}
    X^{\vee\vee\vee\vee} \cong \modobj_{\mathcal{C}}^{\vee} \otimes X \otimes \modobj_{\mathcal{C}}
    \quad (X \in \mathcal{C})
  \end{equation}
  of tensor functors.
\end{theorem}
\begin{proof}
  There is a natural isomorphism as in \eqref{eq:Frob-ten-cat-Radford-S4} induced by
  \begin{equation}
    \label{eq:Radford-iso}
    \Radford_X :=
    \left( X \otimes \modobj_{\mathcal{C}}
      \xrightarrow{\quad \eqref{eq:Nakayama-Frob-tensor-cat-formula} \quad}
      \Nak_{\mathcal{C}}(X^{\vee\vee})
      \xrightarrow{\quad \eqref{eq:Nakayama-Frob-tensor-cat-formula} \quad}
      \modobj_{\mathcal{C}} \otimes X^{\vee\vee\vee\vee} \right)
    \quad (X \in \mathcal{C}).
  \end{equation}
  We note that the isomorphism \eqref{eq:Nakayama-Frob-tensor-cat-formula} used to define \eqref{eq:Radford-iso} comes from the isomorphism \eqref{ew:Nakayama-double-adj} given in Theorem \ref{thm:Nakayama-double-adj}.
  It follows from the coherence property of \eqref{ew:Nakayama-double-adj} that \eqref{eq:Frob-ten-cat-Radford-S4} is indeed an isomorphism of tensor functors.
\end{proof}

We call \eqref{eq:Radford-iso} the {\em Radford isomorphism}.
In some situations, we desire to know an expression of the isomorphism \eqref{eq:Frob-ten-cat-Radford-S4} in much explicit way.
For notational convenience, we rather discuss how the Radford isomorphism $\Radford_X$ is given.
Below we introduce some formulas of $\Radford_X$ from \cite{MR4560996} without expounding the detail of the proof.
The following theorem generalizes \cite[Theorem 8.10.7]{MR3242743}.
  
\begin{theorem}[{\cite[Theorem 5.14]{MR4560996}}]
  \label{thm:Radford-iso-braiding}
  For a braided Frobenius tensor category $\mathcal{C}$ with braiding $\sigma$, the Radford isomorphism is given by
  \begin{equation*}
    \Radford_{X} = (\id_{\modobj} \otimes (u_{X^{\vee}}^{\vee})^{-1}u_{X}) \circ (\sigma_{\modobj, X})^{-1}
  \end{equation*}
  for all $X \in \mathcal{C}$, where $\modobj = \modobj_{\mathcal{C}}$ and
  \begin{equation}
    \label{eq:Drinfeld-iso-def}
    u_X = (\eval_{X} \otimes \id_{X^{\vee\vee}})
    \circ (\sigma_{X,X^{\vee}} \otimes \id_{X^{\vee\vee}})
    \circ (\id_X \otimes \coev_{X^{\vee}})
    \quad (X \in \mathcal{C})
  \end{equation}
  is the Drinfeld isomorphism associated to the braiding $\sigma$.
\end{theorem}

Let $\mathcal{C}$ be a braided tensor category with braiding $\sigma$.
An object $T \in \mathcal{C}$ is said to be {\em transparent} if $\sigma_{X,T} \circ \sigma_{T,X} = \id_{T} \otimes \id_{X}$ for all objects $X \in \mathcal{C}$.
The M\"uger center of $\mathcal{C}$ is the full subcategory of $\mathcal{C}$ consisting of all transparent objects of $\mathcal{C}$.
We say that the M\"uger center of $\mathcal{C}$ is trivial if every transparent object of $\mathcal{C}$ is isomorphic to the direct sum of finitely many copies of $\unitobj$.
It is known that the modular object of a braided finite tensor category is transparent \cite[Corollary 8.10.8]{MR3242743}.
With the help of Theorem~\ref{thm:Radford-iso-braiding}, we can prove the following result by the same way as \cite[Corollary 8.10.8]{MR3242743}.

\begin{corollary}
  The modular object of a braided Frobenius tensor category is transparent.
  Thus, a braided Frobenius tensor category with the trivial M\"uger center is unimodular.
\end{corollary}

Since a cosemisimple coalgebra is a symmetric coalgebra \cite{MR2078404}, a semisimple tensor category is a unimodular Frobenius tensor category whose Nakayama functor is isomorphic to the identity.
We note that the Radford isomorphism $\Radford_X$ is viewed as an isomorphism from $X$ to $X^{\vee\vee\vee\vee}$ in a unimodular Frobenius tensor category.
For a semisimple tensor category, we have:

\newcommand{\trace}{\mathsf{tr}}
\newcommand{\ptrace}{\mathsf{ptr}}

\begin{theorem}[{\cite[Theorem 5.18]{MR4560996}}]
  \label{thm:Radford-iso-semisimple}
  Let $\mathcal{C}$ be a semisimple tensor category such that the endomorphism algebra of every simple object of $\mathcal{C}$ is isomorphic to $\bfk$. Then, for every simple object $X \in \mathcal{C}$, we have
  \begin{equation}
    \label{eq:Radford-iso-semisimple}
    \Radford_X = \frac{\trace((\phi^{-1})^{\vee})}{\trace(\phi)} \cdot \phi^{\vee\vee} \circ \phi,
  \end{equation}
  where $\phi : X \to X^{\vee\vee}$ is an arbitrary isomorphism in $\mathcal{C}$ and $\trace(f) \in \bfk$ is the `trace' for a morphism $f: X \to X^{\vee\vee}$ in $\mathcal{C}$ is defined by the equation
  \begin{equation*}
    \trace(f) \cdot \id_{\unitobj}
    = \eval_{X^{\vee}} \circ (f \otimes \id_{X^{\vee}}) \circ \coev_X.
  \end{equation*}
\end{theorem}

\subsection{Spherical tensor categories}

A {\em pivotal structure} of a tensor category $\mathcal{C}$ is an isomorphism $\id_{\mathcal{C}} \to (-)^{\vee\vee}$ of tensor functors.
Inspired by \cite[Definition 3.5.2]{MR4254952}, we introduce the following definition:

\begin{definition}
  \label{def:spherical-tensor}
  Let $\mathcal{C}$ be a unimodular Frobenius tensor category, and regard the Radford isomorphism $\Radford_X$ ($X \in \mathcal{C}$) as an isomorphism from $X$ to $X^{\vee\vee\vee\vee}$.
  We say that a pivotal structure $p$ of $\mathcal{C}$ is {\em spherical} if the equation $\Radford_X = p_{X^{\vee\vee}} \circ p_X$ holds for all $X \in \mathcal{C}$.
  A {\em spherical tensor category} is a unimodular Frobenius tensor category equipped with a spherical pivotal structure.
\end{definition}

A motivation for introducing spherical {\em finite} tensor categories in \cite{MR4254952} is to look for a `correct' generalization of the notion of spherical fusion category \cite{MR1686423} to a non-semisimple setting.
As the importance of spherical finite tensor categories has become clearer in recent studies (see, {\it e.g.}, \cite{2023arXiv230204509C,2022arXiv220707031F,2021arXiv210313702S,MR4560990}), its `infinite' version Definition \ref{def:spherical-tensor} could also be significant.

We here introduce some criteria for a tensor category to be spherical.
The first one, which concerns semisimple tensor categories, is a generalization of the finite case proved in \cite[Proposition 3.5.4]{MR4254952}.
Thanks to Theorem~\ref{thm:Radford-iso-semisimple}, one can prove Theorem \ref{thm:trace-spherical} below in the completely same way as the finite case.

\begin{theorem}[{\cite[Theorem 5.20]{MR4560996}}]
  \label{thm:trace-spherical}
  Let $\mathcal{C}$ be a semisimple tensor category equipped with a pivotal structure $p$. Then $\mathcal{C}$ is spherical if and only if it is trace spherical  \cite[Definition 3.5.3]{MR4254952}, that is, the equation
  \begin{equation*}
    \trace(p_{X^{\vee}} \circ f^{\vee}) = \trace(p_X \circ f)
  \end{equation*}
  holds for all endomorphisms $f: X \to X$ in $\mathcal{C}$.
\end{theorem}

The second criteria concerns when a unimodular braided Frobenius tensor category is spherical. By Theorem \ref{thm:Radford-iso-braiding} giving a formula of the Radford isomorphism for a braided Frobenius tensor category, one can prove:

\begin{theorem}[{\cite[Theorem 5.21]{MR4560996}}]
  \label{thm:spherical-ribbon}
  Let $\mathcal{C}$ be a unimodular braided Frobenius tensor category equipped with a pivotal structure $p$. Then $\mathcal{C}$ is spherical if and only if the natural isomorphism $\theta := u^{-1} \circ p$ is a twist \cite[Definition 8.10.1]{MR3242743} on $\mathcal{C}$, where $u$ is the Drinfeld isomorphism~\eqref{eq:Drinfeld-iso-def}.
\end{theorem}

\subsection{Modified trace theory}

In recent attempts to construct topological invariants from non-semisimple categories, a modified trace plays a central role.

\begin{definition}
  For a Frobenius tensor category $\mathcal{C}$ equipped with a pivotal structure $p$, a {\em right modified trace} on $\Proj(\mathcal{C})$ is a family of linear maps
  \begin{equation*}
    t = \{ t_P: \Hom_{\mathcal{C}}(P, P) \to \bfk \}_{P \in \Proj(\mathcal{C})}
  \end{equation*}
  satisfying the following two conditions:
  \begin{enumerate}
  \item For any morphisms $f: P \to Q$ and $g: Q \to P$ in $\Proj(\mathcal{C})$, we have
    \begin{equation*}
      t_Q(f g) = t_P(g f).
    \end{equation*}
  \item For any morphism $f: P \otimes X \to P \otimes X$ in $\Proj(\mathcal{C})$ with $P \in \Proj(\mathcal{C})$ and $X \in \mathcal{C}$, we have $t_{P \otimes X}(f) = t_{P}(\mathrm{cl}_X(f))$, where
    \begin{equation*}
      \mathrm{cl}_X(f) = (\id_P \otimes \eval_{X^{\vee}})
      \circ (\id_P \otimes p_X \otimes \id_{X^{\vee}})
      \circ (f \otimes \id_{X^{\vee}}) \circ (\id_P \otimes \coev_X)
    \end{equation*}
    is the morphism obtained from $f$ by `closing' the $X$-strand in the string diagram representing $f$.
  \end{enumerate}
  We note that $\mathcal{C}^{\rev}$ is also a Frobenius pivotal tensor category.
  A {\em left modified trace} on $\Proj(\mathcal{C})$ is a right modified trace on $\Proj(\mathcal{C}^{\rev})$. We say that a left or a right modified trace $t$ on $\Proj(\mathcal{C})$ is {\em non-degenerate} if the pairing
  \begin{equation*}
    \Hom_{\mathcal{C}}(P, Q) \times \Hom_{\mathcal{C}}(Q, P) \to \bfk,
    \quad (f, g) \mapsto t_Q(f g)
  \end{equation*}
  is non-degenerate for all $P, Q \in \Proj(\mathcal{C})$.
\end{definition}

We denote by $t_{\ell}(\mathcal{C})$ and $t_r(\mathcal{C})$ the space of left and right modified traces on $\Proj(\mathcal{C})$, respectively.
In the finite case, it has been observed in \cite{2021arXiv210315772S,2021arXiv210313702S} that modified traces are closely related to the Nakayama functor. By Theorem \ref{thm:Calabi-Yau-1} and the same argument as \cite{2021arXiv210313702S}, one can prove:

\begin{theorem}
  \label{thm:modified-trace}
  For a Frobenius tensor category equipped with a pivotal structure, the following are equivalent:
  \begin{enumerate}
  \item There is a non-zero left modified trace on $\Proj(\mathcal{C})$.
  \item There is a non-zero right modified trace on $\Proj(\mathcal{C})$.
  \item $\mathcal{C}$ is unimodular.
  \end{enumerate}
  Suppose that the above equivalent conditions are satisfied.
  Then both $t_{\ell}(\mathcal{C})$ and $t_r(\mathcal{C})$ are one-dimensional, and every non-zero left or right modified trace on $\Proj(\mathcal{C})$ is non-degenerate.
  Furthermore, $\mathcal{C}$ is spherical if and only if $t_{\ell}(\mathcal{C}) = t_r(\mathcal{C})$.
\end{theorem}

The detail will be given in a forthcoming paper.
A part of this theorem has been known.
Indeed, (3) $\Rightarrow$ (1), (3) $\Rightarrow$ (2) and the uniqueness of modified traces (up to scalar) have been proved in \cite{MR4421818}.

\subsection{Exact sequences of tensor categories}
\label{subsec:exact-seq}

The notion of an exact sequence of Hopf algebras has been effectively used for the study of Hopf algebras.
As a categorical counterpart of this notion, Brugui\`eres and Natale \cite{MR2863377,MR3161401,MR4281372} introduced and studied exact sequences of tensor categories.
We exhibit an application of our results to exact sequences of tensor categories.
To begin with, we prove:

\begin{theorem}[{\cite[Theorem 3.6]{2023arXiv230314687S}}]
  \label{thm:intro-co-Fb-1}
  Let $F: \mathcal{C} \to \mathcal{D}$ be a tensor functor between tensor categories $\mathcal{C}$ and $\mathcal{D}$, and let $G$ be a right adjoint of $\Ind(F)$.
  Then we have:
  \begin{enumerate}
  \item [(a)] If $\mathcal{C}$ is Frobenius and $G$ is exact, then $\mathcal{D}$ is Frobenius.
  \item [(b)] If $\mathcal{D}$ is Frobenius and $G(\mathcal{D}) \subset \mathcal{C}$, then $\mathcal{C}$ is Frobenius.
  \end{enumerate}
\end{theorem}
\begin{proof}
  (a) Suppose that $\mathcal{C}$ is Frobenius and $G$ is exact.
  Then $\mathcal{C}$ has a non-zero projective object, say $P$.
  Since a tensor functor is faithful, we have $F(P) \ne 0$.
  Furthermore, $F(P)$ is a projective object of $\mathcal{D}$, since 
  \begin{equation*}
    \Hom_{\mathcal{D}}(F(P), -) \cong \Hom_{\Ind(\mathcal{C})}(P, -) \circ G
  \end{equation*}
  is exact. Hence $\mathcal{D}$ is Frobenius.

  (b) Suppose that $\mathcal{D}$ is Frobenius and $G(\mathcal{D}) \subset \mathcal{C}$.
  Then the unit object $\unitobj \in \mathcal{D}$ has an injective hull in $\mathcal{D}$, say $E$.
  A tensor functor is exact by definition.
  The object $G(E) \in \mathcal{C}$ is a non-zero injective object, since
  \begin{equation*}
    \Hom_{\mathcal{C}}(-, G(E)) \cong \Hom_{\mathcal{D}}(-, E) \circ F
  \end{equation*}
  is exact and $\Hom_{\mathcal{C}}(\unitobj, G(E)) \cong \Hom_{\mathcal{D}}(\unitobj, E) \ne 0$. Hence $\mathcal{C}$ is Frobenius.
\end{proof}

Let $\mathcal{C}' \xrightarrow{\ \iota \ } \mathcal{C} \xrightarrow{\ F \ } \mathcal{D}$ be a sequence of tensor functors between tensor categories $\mathcal{C}$, $\mathcal{C}'$ and $\mathcal{D}$. We say that this sequence is {\em exact} \cite{MR2863377,MR2863377,MR4281372} if the following conditions (1), (2), (3) and (4) are satisfied:
\begin{enumerate}
\item The functor $\iota$ is fully faithful.
\item The functor $F$ is {\em dominant} in the sense that every object $Y \in \mathcal{D}$ is a subobject of an object of the form $F(X)$ for some $X \in \mathcal{C}$.
\item The functor $F$ is {\em normal} in the sense that every object $X \in \mathcal{C}$ has a subobject such that $F(X_0)$ is the maximal trivial subobject of $F(X)$ (here an object of $\mathcal{D}$ is said to be {\em trivial} if it is isomorphic to the direct sum of finitely many copies of the unit object).
\item The kernel of $F$, defined by $\Ker(F) := \{ X \in \mathcal{C} \mid \text{$F(X)$ is trivial} \}$, coincides with the essential image of the functor $\iota$.
\end{enumerate}

Let $G: \Ind(\mathcal{D}) \to \Ind(\mathcal{C})$ be a right adjoint of $\Ind(F)$.
The conditions (2) and (3) above are characterized in terms of $G$ as follows:
The functor $F$ is dominant if and only if $G$ is faithful \cite[Lemma 3.1]{MR2863377}.
The functor $F$ is normal if and only if $G(\unitobj)$ is a filtered colimit of trivial objects (see \cite[Proposition 3.5]{MR2863377} and \cite[Lemma 3.8]{2023arXiv230314687S}).

It is known that the class of co-Frobenius Hopf algebras is closed under exact sequences \cite{MR3032811}.
Natale \cite{MR4281372} asked whether the class of Frobenius tensor categories is closed under exact sequences. We have given an affirmative answer to this question as follows:

\begin{theorem}[{\cite[Theorem 3.9]{2023arXiv230314687S}}]
  \label{thm:intro-co-Fb-2}
  Let $\mathcal{C}' \xrightarrow{\ \iota \ } \mathcal{C} \xrightarrow{\ F \ } \mathcal{C}''$ be an exact sequence of tensor categories such that a right adjoint of $\Ind(F)$ is exact.
  Then $\mathcal{C}$ is Frobenius if and only if both $\mathcal{C}'$ and $\mathcal{C}''$ are.
\end{theorem}

\section{Applications to coquasi-bialgebras}
\label{sec:coquasibialg}

\subsection{Coquasi-bialgebras with preantipode}

A {\em coquasi-bialgebra} \cite{MR1187289,MR1887584} is a coalgebra $H$ equipped with two coalgebra maps $m: H \otimes_{\bfk} H \to H$ and $u: \bfk \to H$ and a linear map $\omega: H \otimes_{\bfk} H \otimes_{\bfk} H \to \bfk$ such that $\omega$ is invertible with respect to the convolution product and the equations
\begin{gather*}
  \omega(\id_H \otimes \id_H \otimes m) * \omega(m \otimes \id_H \otimes \id_H)
  = (\varepsilon \otimes \omega) * \omega(\id_H \otimes m \otimes \id_H) * (\omega \otimes \varepsilon), \\
  \omega(\id_H \otimes u \otimes \id_H) = \varepsilon \otimes \varepsilon,
  \quad m(u \otimes \id_H) = \id_H = m(\id_H \otimes u), \\
  m(\id_H \otimes m) * \omega = \omega * m(m \otimes \id_H)
\end{gather*}
hold, where $*$ denote the convolution product.
This much complicated notion is natural from the viewpoint of Tannaka-Krein reconstruction:
If $H$ is a coquasi-bialgebra, then ${}^H\Mod_{\fd}$ is a monoidal category with respect to the tensor product over $\bfk$ and the associator given by $\omega$.
Moreover, a coalgebra $C$ has a structure of a coquasi-bialgebra if and only if ${}^C\Mod_{\fd}$ has a structure of a monoidal category such that the forgetful functor ${}^C\Mod_{\fd} \to \Vect$ is quasi-monoidal \cite{MR1187289,MR1887584}.

A coquasi-bialgebra is the dual notion of a quasi-bialgebra introduced by Drinfeld in \cite{MR1025154}. As an analogue of the notion of a Hopf algebra, a {\em quasi-Hopf algebra} was also introduced in \cite{MR1025154}.
A {\em coquasi-Hopf algebra}, which is the dual notion of a quasi-Hopf algebra, is a coquasi-bialgebra $(H, m, u)$ equipped with a triple $(s, \alpha, \beta)$ consisting of an anti-coalgebra map $s: H \to H$ and two linear maps $\alpha, \beta: H \to \bfk$ satisfying the equations
\begin{gather*}
  h_{(1)} \beta(h_{(2)}) s(h_{(3)}) = \beta(h) 1_H, \quad
  s(h_{(1)}) \alpha(h_{(2)}) h_{(3)} = \alpha(h) 1_H, \\
  \omega^{-1}(s(h_{(1)}) \otimes \alpha(h_{(2)}) h_{(3)} \beta(h_{(4)}) \otimes s(h_{(5)})) \\
  = \varepsilon(h)
  = \omega(h_{(1)} \otimes \beta(h_{(2)})s(h_{(3)})\alpha(h_{(4)}) \otimes h_{(5)})
\end{gather*}
for all $h \in H$, where $x y = m(x \otimes y)$ ($x, y \in H$), $1_H = u(1_{\bfk})$ and $\omega^{-1}$ is the inverse of $\omega$ with respect to the convolution product.
The triple $(s, \alpha, \beta)$ is called a {\em coquasi-antipode} of $H$.
It is known that ${}^H\Mod_{\fd}$ is right rigid if $H$ is a coquasi-Hopf algebra.
However, the converse does not hold in general, as has been pointed out by Schauenburg in \cite{MR1913442}.

According to Schauenburg \cite[Theorem 2.6]{MR2037710}, the monoidal category ${}^H\Mod_{\fd}$ for a coquasi-bialgebra $H$ is right rigid if and only if a certain structure theorem for Hopf bimodules hold. Ardizzoni and Pavarin \cite{MR3004083} showed that the latter condition for $H$ is equivalent to that $H$ has a {\em preantipode}, that is, a linear map $S: H \to H$ satisfying the equations
\begin{gather*}
  S(h_{(1)})_{(1)} h_{(2)} \otimes S(h_{(1)})_{(2)} = 1_H \otimes S(h), \\
  S(h_{(2)})_{(1)} \otimes h_{(1)} S(h_{(2)})_{(2)} = S(h) \otimes 1_H, \\
  \omega(h_{(1)} \otimes S(h_{(2)}) \otimes h_{(3)}) = \varepsilon(h)
\end{gather*}
for all $h \in H$.
In summary, the monoidal category ${}^H\Mod_{\fd}$ is right rigid if and only if $H$ has a preantipode.
We also remark that a coquasi-Hopf algebra $H$ has a preantipode given by $S(h) = \beta(h_{(1)}) s(h_{(2)}) \alpha(h_{(3)})$ for $h \in H$, where $(s, \alpha, \beta)$ is the coquasi-antipode of $H$ \cite[Theorem 3.10]{MR3004083}.

Saracco \cite{MR4207391} investigated how a preantipode is reconstructed from the forgetful functor ${}^H\Mod_{\fd} \to \Vect$ when ${}^H\Mod_{\fd}$ is right rigid.
Moreover, an explicit construction of a dual object was given:
Following \cite[Remark 3.20]{MR4207391}, a right dual of $X \in {}^H\Mod_{\fd}$ is given by ${}^{\vee\!}X := (X^* \otimes_{\bfk} H)^{\mathop{\mathrm{co}}H}$ as a vector space, where $(-)^{\mathop{\mathrm{co}}H}$ means the coinvariant as a right $H$-comodule.
In general, ${}^{\vee\!}X$ may not be isomorphic to $X^*$. This is a crucial difference to the case where $H$ is a coquasi-Hopf algebra.

\subsection{QcF property of a coquasi-bialgebra with preantipode}

Let $H$ be a coquasi-bialgebra with preantipode.
As in the case of Hopf algebras, we introduce:

\begin{definition}
  A left (right) cointegral on $H$ is a homomorphism $H \to \bfk$ of left (right) comodules, where $\bfk$ is the trivial $H$-comodule.
\end{definition}

As we have noted in the previous subsection, ${}^H\Mod_{\fd}$ is a right rigid monoidal category with respect to the tensor product over $\bfk$.
Thus the category $\mathcal{C} := ({}^H\Mod_{\fd})^{\rev}$ satisfies Assumption \ref{assumption:pretensor-cat}. By applying our results to $\mathcal{C}$, we obtain:

\begin{theorem}[see \cite{MR1786197} for Hopf algebras and \cite{MR1995128} for coquasi-Hopf algebras]
  \label{thm:coquasi-bialg-QcF-1}
  For a coquasi-bialgebra $H$ with preantipode, the following are equivalent:
  \begin{enumerate}
  \item $H$ is QcF (as a coalgebra; the same applies below).
  \item $H$ is left semiperfect.
  \item $H$ has a finite-dimensional non-zero projective left comodule.
  \item $H$ has a finite-dimensional non-zero injective left comodule.
  \item The rational part of $H^*$ as a left $H^*$-module is non-zero.
  \item There is a non-zero left integral on $H$.
  \item Left-right switched versions of (2), (3), (4), (5) and (6).
  \end{enumerate}
\end{theorem}
\begin{proof}
  By applying Theorems~\ref{thm:one-sided-rigidity} to $\mathcal{C} := ({}^H\Mod_{\fd})^{\rev}$ or $\mathcal{C}^{\op,\rev}$, we see that (1), (2), (3), (4), (5) and their left-right switched versions are equivalent. Moreover, Theorem~\ref{thm:one-sided-rigidity} shows that (6) and its left-right switched version imply (1).
  By the fact that a QcF coalgebra $Q$ is a generator both in ${}^Q\Mod$ and $\Mod^Q$ \cite[Corollary 3.3.11]{MR1786197}, we see that (1) implies both (6) and its left-right switched version. The proof is done.
\end{proof}

By Theorems \ref{thm:Nakayama-cF} and~\ref{thm:coquasi-bialg-QcF-1}, we have:

\begin{theorem}
  \label{thm:coquasi-bialg-QcF-2}
  Let $H$ be a coquasi-bialgebra with preantipode.
  We assume that $H$ satisfies the equivalent conditions of Theorem~\ref{thm:coquasi-bialg-QcF-1}. Then the following are equivalent:
  \begin{enumerate}
  \item $H$ is co-Frobenius.
  \item $\dim_{\bfk} ({}^{\vee\vee\!}X) = \dim_{\bfk}(X)$ for all simple $X \in {}^H\Mod_{\fd}$.
  \end{enumerate}
\end{theorem}

In the case where $H$ is a coquasi-Hopf algebra, the dual object of $X \in {}^H\Mod_{\fd}$ is built on the dual space $X^*$.
Thus we obtain the following consequence:
If a coquasi-Hopf algebra $H$ satisfies the equivalent conditions of Theorem~\ref{thm:coquasi-bialg-QcF-1}, then $H$ is co-Frobenius ({\it cf}. \cite[Theorem 4.5]{MR1995128}).
A natural question arises:

\begin{question}
  Is there a QcF coquasi-bialgebra with preantipode that is not co-Frobenius?
\end{question}

\subsection{De-equivariantization}

The equivariantization is a kind of category of `fixed points' of a category on which a group acts, and the de-equivariantization can be thought of as the inverse of equivariantization \cite{MR3242743}.
At the end of this paper, we discuss coquasi-bialgebras constructed by de-equivariantization by affine group schemes in \cite{MR3160718}.
A starting observation by Brugui\`eres and Natale \cite{MR2863377,MR3161401} is that the de-equivariantization of a finite tensor category by a finite group fits into a certain exact sequence of tensor categories.
Their observation has been extended in \cite{2023arXiv230314687S} to de-equivariantization by affine group schemes as follows.

Given a monoidal category $\mathcal{M}$, we denote by $\mathcal{Z}(\mathcal{M})$ its Drinfeld center.
Let $G$ be an affine group scheme, and let $\mathcal{C}$ be a tensor category equipped with a central embedding of $\Rep(G) := {}^{\mathcal{O}(G)}\Mod_{\fd}$ into $\mathcal{C}$, that is, a fully faithful tensor functor $\iota: \Rep(G) \to \mathcal{C}$ such that there exists a braided tensor functor $\widetilde{\iota} : \Rep(G) \to \mathcal{Z}(\mathcal{C})$ such that $\iota = U \circ \widetilde{\iota}$, where $U: \mathcal{Z}(\mathcal{C}) \to \mathcal{C}$ is the forgetful functor.
Then the coordinate algebra $\mathcal{O}(G)$ becomes a commutative algebra in $\mathcal{Z}(\Ind(\mathcal{C}))$.
The category of left $\mathcal{O}(G)$-modules in ${}_{\mathcal{O}(G)}\Ind(\mathcal{C})$ is an abelian monoidal category with respect to the tensor product over $\mathcal{O}(G)$.
We recall that an object $X$ of a category $\mathcal{A}$ is said to be {\em finitely generated} if the functor $\Hom_{\mathcal{A}}(X, -)$ preserves filtered colimits of monomorphisms.
Now let $\mathcal{D}$ denote the full subcategory of ${}_{\mathcal{O}(G)}\Ind(\mathcal{C})$ consisting of finitely generated objects.
It can be shown that $\mathcal{D}$ is closed under $\otimes_{\mathcal{O}(G)}$. Hence, $\mathcal{D}$ is a monoidal category.

\begin{theorem}[\cite{2023arXiv230314687S}]
  \label{thm:de-equiv-Frobenius}
  Let $G$, $\mathcal{C}$ and $\mathcal{D}$ be as above.
  Suppose that the monoidal category $\mathcal{D}$ is a tensor category. Then there is an exact sequence
  \begin{equation*}
    \Rep(G) \to \mathcal{C} \to \mathcal{D}
  \end{equation*}
  of tensor categories.
  Hence, by Theorem \ref{thm:intro-co-Fb-2}, the tensor category $\mathcal{C}$ is Frobenius if and only if the Hopf algebra $\mathcal{O}(G)$ is co-Frobenius and the tensor category $\mathcal{D}$ is Frobenius.
\end{theorem}

We consider the case where $\mathcal{C} = {}^H\Mod_{\fd}$ for some Hopf algebra $H$ with bijective antipode.
Let $(K, r)$ be a braided central Hopf subalgebra of $H$ \cite[Definition 3.1]{MR3160718}, that is, a pair of a Hopf subalgebra $K$ of $H$ and a bilinear map $r: H \times K \to \bfk$ satisfying certain conditions similar to the axioms for universal r-forms.
Since $K$ is commutative \cite[Remark 3.2]{MR3160718}, we may assume that $K = \mathcal{O}(G)$ for some affine group scheme $G$.
As shown in \cite{MR3160718}, the pair $(K, r)$ gives rise to a central embedding $\Rep(G) \to \mathcal{C}$.

We now assume that there is a convolution-invertible $\mathcal{O}(G)$-linear map from $H$ to $\mathcal{O}(G)$ preserving the unit and the counit.
The main result of \cite{MR3160718} states that the quotient coalgebra $Q = H/\mathcal{O}(G)^{+}H$, where $\mathcal{O}(G)^{+}$ is the kernel of the counit of $\mathcal{O}(G)$, has a structure of a coquasi-bialgebra such that ${}^{Q}\Mod$ is equivalent to ${}_{\mathcal{O}(G)}\Ind(\mathcal{C})$ as a monoidal category.
Theorem \ref{thm:de-equiv-Frobenius} gives the following consequences: Suppose that the monoidal category ${}^Q\Mod_{\fd}$ is rigid. Then the coquasi-bialgebra $Q$ is QcF if and only if both $\mathcal{O}(G)$ and $H$ are co-Frobenius.

\bibliographystyle{amsplain}

\begin{thebibliography}{10}

\bibitem{MR3032811}
Nicol\'{a}s Andruskiewitsch and Juan Cuadra, \emph{On the structure of
  (co-{F}robenius) {H}opf algebras}, J. Noncommut. Geom. \textbf{7} (2013),
  no.~1, 83--104. \MR{3032811}

\bibitem{MR3410615}
Nicol\'{a}s Andruskiewitsch, Juan Cuadra, and Pavel Etingof, \emph{On two
  finiteness conditions for {H}opf algebras with nonzero integral}, Ann. Sc.
  Norm. Super. Pisa Cl. Sci. (5) \textbf{14} (2015), no.~2, 401--440.
  \MR{3410615}

\bibitem{MR3160718}
Iv\'{a}n Angiono, C\'{e}sar Galindo, and Mariana Pereira,
  \emph{De-equivariantization of {H}opf algebras}, Algebr. Represent. Theory
  \textbf{17} (2014), no.~1, 161--180.

\bibitem{MR3004083}
Alessandro Ardizzoni and Alice Pavarin, \emph{Preantipodes for dual
  quasi-bialgebras}, Israel J. Math. \textbf{192} (2012), no.~1, 281--295.

\bibitem{MR1686423}
John~W. Barrett and Bruce~W. Westbury, \emph{Spherical categories}, Adv. Math.
  \textbf{143} (1999), no.~2, 357--375. \MR{1686423}

\bibitem{MR2278058}
Margaret Beattie, Daniel Bulacu, and Blas Torrecillas, \emph{Radford's {$S^4$}
  formula for co-{F}robenius {H}opf algebras}, J. Algebra \textbf{307} (2007),
  no.~1, 330--342. \MR{2278058}

\bibitem{MR2863377}
Alain Brugui\`eres and Sonia Natale, \emph{Exact sequences of tensor
  categories}, Int. Math. Res. Not. IMRN (2011), no.~24, 5644--5705.

\bibitem{MR3161401}
\bysame, \emph{Central exact sequences of tensor categories, equivariantization
  and applications}, J. Math. Soc. Japan \textbf{66} (2014), no.~1, 257--287.
  \MR{3161401}

\bibitem{MR2869176}
Alain Brugui\`eres and Alexis Virelizier, \emph{Quantum double of {H}opf monads
  and categorical centers}, Trans. Amer. Math. Soc. \textbf{364} (2012), no.~3,
  1225--1279. \MR{2869176}

\bibitem{MR1995128}
D.~Bulacu and S.~Caenepeel, \emph{Integrals for (dual) quasi-{H}opf algebras.
  {A}pplications}, J. Algebra \textbf{266} (2003), no.~2, 552--583.

\bibitem{MR2078404}
F.~Casta\~{n}o Iglesias, S.~D\u{a}sc\u{a}lescu, and C.~N\u{a}st\u{a}sescu,
  \emph{Symmetric coalgebras}, J. Algebra \textbf{279} (2004), no.~1, 326--344.
  \MR{2078404}

\bibitem{MR2684139}
William Chin and Daniel Simson, \emph{Coxeter transformation and inverses of
  {C}artan matrices for coalgebras}, J. Algebra \textbf{324} (2010), no.~9,
  2223--2248. \MR{2684139}

\bibitem{2023arXiv230204509C}
Francesco {Costantino}, Nathan {Geer}, Bertrand {Patureau-Mirand}, and Alexis
  {Virelizier}, \emph{{Non compact (2+1)-TQFTs from non-semisimple spherical
  categories}}, arXiv e-prints (2023), arXiv:2302.04509.

\bibitem{MR1904645}
J.~Cuadra and J.~G\'{o}mez-Torrecillas, \emph{Idempotents and
  {M}orita-{T}akeuchi theory}, Comm. Algebra \textbf{30} (2002), no.~5,
  2405--2426. \MR{1904645}

\bibitem{MR4254952}
Christopher~L. Douglas, Christopher Schommer-Pries, and Noah Snyder,
  \emph{Dualizable tensor categories}, Mem. Amer. Math. Soc. \textbf{268}
  (2020), no.~1308, vii+88. \MR{4254952}

\bibitem{MR1025154}
V.~G. Drinfel\cprime~d, \emph{Almost cocommutative {H}opf algebras}, Algebra i
  Analiz \textbf{1} (1989), no.~2, 30--46.

\bibitem{MR1786197}
Sorin D\u{a}sc\u{a}lescu, Constantin N\u{a}st\u{a}sescu, and \c{S}erban Raianu,
  \emph{Hopf algebras}, Monographs and Textbooks in Pure and Applied
  Mathematics, vol. 235, Marcel Dekker, Inc., New York, 2001, An introduction.
  \MR{1786197}

\bibitem{MR3242743}
Pavel Etingof, Shlomo Gelaki, Dmitri Nikshych, and Victor Ostrik, \emph{Tensor
  categories}, Mathematical Surveys and Monographs, vol. 205, American
  Mathematical Society, Providence, RI, 2015. \MR{3242743}

\bibitem{MR2097289}
Pavel Etingof, Dmitri Nikshych, and Viktor Ostrik, \emph{An analogue of
  {R}adford's {$S^4$} formula for finite tensor categories}, Int. Math. Res.
  Not. (2004), no.~54, 2915--2933. \MR{2097289}

\bibitem{2022arXiv220707031F}
J{\"u}rgen {Fuchs}, C{\'e}sar {Galindo}, David {Jaklitsch}, and Christoph
  {Schweigert}, \emph{{Spherical Morita contexts and relative Serre functors}},
  arXiv e-prints (2022), arXiv:2207.07031.

\bibitem{MR4042867}
J\"{u}rgen Fuchs, Gregor Schaumann, and Christoph Schweigert,
  \emph{Eilenberg-{W}atts calculus for finite categories and a bimodule
  {R}adford {$S^4$} theorem}, Trans. Amer. Math. Soc. \textbf{373} (2020),
  no.~1, 1--40. \MR{4042867}

\bibitem{MR4403279}
\bysame, \emph{A modular functor from state sums for finite tensor categories
  and their bimodules}, Theory Appl. Categ. \textbf{38} (2022), Paper No. 15,
  436--594. \MR{4403279}

\bibitem{MR4586249}
J\"{u}rgen Fuchs and Christoph Schweigert, \emph{Internal natural
  transformations and frobenius algebras in the drinfeld center}, Transform.
  Groups \textbf{28} (2023), no.~2, 733--768.

\bibitem{MR4421818}
Nathan Geer, Jonathan Kujawa, and Bertrand Patureau-Mirand, \emph{M-traces in
  (non-unimodular) pivotal categories}, Algebr. Represent. Theory \textbf{25}
  (2022), no.~3, 759--776.

\bibitem{MR2253657}
Miodrag~Cristian Iovanov, \emph{Co-{F}robenius coalgebras}, J. Algebra
  \textbf{303} (2006), no.~1, 146--153. \MR{2253657}

\bibitem{MR3125851}
\bysame, \emph{Abstract algebraic integrals and {F}robenius categories},
  Internat. J. Math. \textbf{24} (2013), no.~10, 1350081, 36. \MR{3125851}

\bibitem{MR3150709}
\bysame, \emph{Generalized {F}robenius algebras and {H}opf algebras}, Canad. J.
  Math. \textbf{66} (2014), no.~1, 205--240. \MR{3150709}

\bibitem{MR2182076}
Masaki Kashiwara and Pierre Schapira, \emph{Categories and sheaves},
  Grundlehren der Mathematischen Wissenschaften [Fundamental Principles of
  Mathematical Sciences], vol. 332, Springer-Verlag, Berlin, 2006. \MR{2182076}

\bibitem{MR1712872}
Saunders Mac~Lane, \emph{Categories for the working mathematician}, second ed.,
  Graduate Texts in Mathematics, vol.~5, Springer-Verlag, New York, 1998.
  \MR{1712872}

\bibitem{MR1187289}
Shahn Majid, \emph{Tannaka-{K}re\u{\i}n theorem for quasi-{H}opf algebras and
  other results}, Deformation theory and quantum groups with applications to
  mathematical physics ({A}mherst, {MA}, 1990), Contemp. Math., vol. 134, Amer.
  Math. Soc., Providence, RI, 1992, pp.~219--232.

\bibitem{MR4281372}
Sonia Natale, \emph{On the notion of exact sequence: from {H}opf algebras to
  tensor categories}, Hopf algebras, tensor categories and related topics,
  Contemp. Math., vol. 771, Amer. Math. Soc., [Providence], RI, [2021]
  \copyright 2021, pp.~225--254. \MR{4281372}

\bibitem{MR4207391}
Paolo Saracco, \emph{Coquasi-bialgebras with preantipode and rigid monoidal
  categories}, Algebr. Represent. Theory \textbf{24} (2021), no.~1, 55--80.

\bibitem{MR1913442}
Peter Schauenburg, \emph{Hopf algebra extensions and monoidal categories}, New
  directions in {H}opf algebras, Math. Sci. Res. Inst. Publ., vol.~43,
  Cambridge Univ. Press, Cambridge, 2002, pp.~321--381.

\bibitem{MR1887584}
\bysame, \emph{Hopf bimodules, coquasibialgebras, and an exact sequence of
  {K}ac}, Adv. Math. \textbf{165} (2002), no.~2, 194--263.

\bibitem{MR2037710}
\bysame, \emph{Two characterizations of finite quasi-{H}opf algebras}, J.
  Algebra \textbf{273} (2004), no.~2, 538--550.

\bibitem{2021arXiv210315772S}
Christoph {Schweigert} and Lukas {Woike}, \emph{{The Trace Field Theory of a
  Finite Tensor Category}}, Algebras and Representation Theory.

\bibitem{2021arXiv210313702S}
  Taiki {Shibata} and Kenichi {Shimizu}, \emph{{Modified traces and the Nakayama
      functor}},
  Algebr. Represent. Theory 26 (2023), no. 2, 513--551.

\bibitem{2023arXiv230314687S}
\bysame, \emph{{Exact sequences of Frobenius tensor categories}}, arXiv
  e-prints (2023), arXiv:2303.14687.

\bibitem{MR4560996}
Taiki Shibata and Kenichi Shimizu, \emph{Nakayama functors for coalgebras and
  their applications to {F}robenius tensor categories}, Adv. Math. \textbf{419}
  (2023), Paper No. 108960, 61.

\bibitem{2019arXiv190400376S}
Kenichi {Shimizu}, \emph{{Relative Serre functor for comodule algebras}}, arXiv
  e-prints (2019), arXiv:1904.00376.

\bibitem{2022arXiv220808203S}
\bysame, \emph{{Nakayama functor for monads on finite abelian categories}},
  arXiv e-prints (2022), arXiv:2208.08203.

\bibitem{MR4560990}
Kenichi Shimizu, \emph{Ribbon structures of the {D}rinfeld center of a finite
  tensor category}, Kodai Math. J. \textbf{46} (2023), no.~1, 75--114.

\bibitem{MR292876}
Mitsuhiro Takeuchi, \emph{Free {H}opf algebras generated by coalgebras}, J.
  Math. Soc. Japan \textbf{23} (1971), 561--582. \MR{292876}

\bibitem{MR472967}
\bysame, \emph{Morita theorems for categories of comodules}, J. Fac. Sci. Univ.
  Tokyo Sect. IA Math. \textbf{24} (1977), no.~3, 629--644. \MR{472967}

\bibitem{MR1098991}
K.-H. Ulbrich, \emph{On {H}opf algebras and rigid monoidal categories},
  vol.~72, 1990, Hopf algebras, pp.~252--256. \MR{1098991}

\bibitem{2023arXiv230206192Y}
Harshit {Yadav}, \emph{{On Unimodular module categories}}, arXiv e-prints
  (2023), arXiv:2302.06192.

\end{thebibliography}
\def\cprime{$'$}
\providecommand{\bysame}{\leavevmode\hbox to3em{\hrulefill}\thinspace}
\providecommand{\MR}{\relax\ifhmode\unskip\space\fi MR }
\providecommand{\MRhref}[2]{%
  \href{http://www.ams.org/mathscinet-getitem?mr=#1}{#2}
}
\providecommand{\href}[2]{#2}

\end{document}